\documentclass[10pt,reqno]{amsart}
\usepackage{hyperref}\hypersetup{colorlinks=true, citecolor=blue}
\usepackage{graphicx}
\usepackage{xfrac}
\usepackage{esint,amsfonts,amsmath,amssymb,epsfig,mathrsfs,bm}
\usepackage{ifthen,color,mathtools}
\newtheorem{theorem}{Theorem}[section]
\newtheorem{lemma}[theorem]{Lemma}
\newtheorem{proposition}[theorem]{Proposition}
\newtheorem{corollary}[theorem]{Corollary}

\theoremstyle{definition}
\newtheorem{definition}[theorem]{Definition}

\theoremstyle{remark}
\newtheorem{remark}[theorem]{Remark}
\theoremstyle{remark}

\numberwithin{equation}{section}

\newcommand{\R}{\mathbb{R}}
\newcommand{\eps}{{\varepsilon}}

\newcommand{\dist}{{\textup {dist}}}

\newcommand{\de}{\partial}

\renewcommand{\d}{{\rm d}}
\newcommand{\Id}{{\rm Id}}
\newcommand{\spt}{{\rm spt}\,}

\newcommand{\loc}{\textup{loc}}

\renewcommand{\and}{\quad \text{and} \quad}

\renewcommand{\div}{\textup{div}}

\renewcommand{\top}{\textup{top}}
\newcommand{\low}{\textup{low}}

\newcommand{\cL}{{\mathcal{L}}}
\newcommand{\cH}{{\mathcal{H}}}

\newcommand{\LL}{\mathbb{A}^{\sfrac12}}

\newcommand\N{{\mathbb N}}

\newcommand{\ie}{\textit{i.e.}}

\newcommand{\mA}{\mathbb{A}}
\newcommand{\mB}{\mathbb{B}}
\newcommand{\PPhi}{\Phi_{x_0}}

\newcommand{\Ico}{\mathscr{I}_u}
\newcommand{\Hco}{\mathscr{H}_u}
\newcommand{\Eco}{\mathscr{E}_u}
\newcommand{\Tr}{\operatorname{Tr}}
\newcommand\MF[1]{{\color{blue} #1}}

\title[Thin obstacles problems with Sobolev variable coefficients]
{On the free boundary for thin obstacle problems with Sobolev variable coefficients}

\author[G.~Andreucci]{Giovanna Andreucci}
\address{Dipartimento di Matematica, Sapienza Universit\`a di Roma}
\curraddr{P.le Aldo Moro 5, I-00185 Roma}
\email{giovanna.andreucci@uniroma1.it}
\author[M.~Focardi]{Matteo Focardi}
\address{DiMaI, Universit\`a degli Studi di Firenze}
\curraddr{Viale Morgagni 67/A, I-50134 Firenze}
\email{matteo.focardi@unifi.it}
\author[E.~Spadaro]{Emanuele Spadaro}
\address{Dipartimento di Matematica, Sapienza Universit\`a di Roma}
\curraddr{P.le Aldo Moro 5, I-00185 Roma}
\email{spadaro@mat.uniroma1.it}
\thanks{E.~S.~has been supported the ERC-STG Grant n. 759229 HiCoS “Higher Co-dimension Singularities: Minimal Surfaces and the Thin Obstacle Problem”. M.~F. is a member of GNAMPA.}


\date{}
\begin{document}
\begin{abstract}
We establish a quasi-monotonicity formula {for an intrinsic frequency function related to solutions to} thin obstacle problems with zero obstacle driven by quadratic energies with Sobolev $W^{1,p}$ coefficients, with $p$ bigger than the space dimension.
From this we deduce several regularity and structural properties of the corresponding free boundaries at those distinguished points with finite order of contact with the obstacle.
In particular, we prove the rectifiability {and the local finiteness of the Minkowski content} of the whole free boundary in the case of Lipschitz coefficients.
\end{abstract}

\maketitle

%
%
\section{Introduction}

In this article we consider a class of
lower dimensional obstacle problems with variable coefficients which has been extensively considered in the literature.
In order to state the results, we introduce the following notation:
for any subset $E \subset \R^{n+1}$ we set
\[
E^+ := E\cap\big\{x\in\R^{n+1}:\,x_{n+1}>0\big\}
\and
E' := E \cap \big\{x_{n+1}=0\big\}.
\]
For any point $x \in \R^{n+1}$ we will write $x = (x', x_{n+1})\in \R^{n} 
\times \R$.
Moreover, $B_r(x) \subset \R^{n+1}$ denotes the open ball centered
at $x\in \R^{n+1}$ with radius $r>0$, and $\overline{B_r}(x)$ its closure
(we omit to write the point $x$ if the origin).

\medskip

We consider the problem of minimizing a variable coefficient quadratic (\textit{Dirichlet}) energy with an unilateral constraint:
\begin{equation}\label{e:minimization}
\min_{v\in\mathscr{A}
} \int_{B_1^+} \langle\mA(x)\nabla v,\nabla v\rangle \d x,
\end{equation}
where the class of competing functions is given by 
$$
\mathscr{A}:=\Big\{v\in H^1(B_1^+):\,v = g \;\textup{in }\,(\de B_1)^+
\and v\geq 0 \;\textup{in }\,B_1'\Big\},
$$
with $g\in H^{\frac12}((\de B_1)^+)$ such that $\mathcal{A}$ is not empty (the boundary conditions are meant in the sense of traces).
We assume the following hypotheses: 
\begin{itemize}
 \item[(H1)] $\mA\in W^{1,p}(B_1,\R^{(n+1)\times (n+1)})$, $p\in (n+1,\infty]$ ({in particular, by Morrey's embedding, we have $\mA\in C^{0,\alpha}(B_1,\R^{(n+1)\times (n+1)})$, 
$\alpha:=1-\frac{n+1}p$});
 \item[(H2)]  {$\mA(x)=(a_{ij}(x))_{i,j=1}^{n+1}$ is a symmetric, bounded and coercive matrix, 
i.e. for every $x\in B_1^+$, $i,j\in\{1,\ldots,n+1\}$, and $\xi\in\R^{n+1}$
$a_{ij}(x)=a_{ji}(x)$, and}
 \[
 \lambda|\xi|^2\leq  \langle\mA(x)\xi,\xi\rangle\leq\Lambda|\xi|^2\,,
 \]
for some $0<\lambda\leq\Lambda$,
\end{itemize}

\medskip

By a well-known result due to Ural'tseva \cite{Ural87}, the minimizers of \eqref{e:minimization} are $C^{1,\beta}$ regular for some $\beta>0$ and satisfy an elliptic partial differential equation in the half ball $B_1^+$, whereas on the flat part of the boundary $B_1'$ the unilateral constraint $u\geq 0$ leads to a free boundary problem. The Euler--Lagrange equation of \eqref{e:minimization} is
\begin{equation}\label{e:ob-pb local intro}
\begin{cases}
\div\big(\mA(x)\nabla u(x)\big) = 0  & \text{in }\; 
B_1^+,\\
u \geq 0, \quad \mA\nabla u\cdot e_{n+1}\leq 0,\quad
u(\mA\nabla u\cdot e_{n+1})=0   & \text{in }\; B_1'.
\end{cases}
\end{equation}
The condition $u(\mA\nabla u\cdot e_{n+1})=0$ in $B_1'$ is the so called \textit{Signorini complementary} boundary condition.
The behaviour of $u$ on $B_1'$ is not prescribed and is characterized by the so called \textit{free boundary} $\Gamma(u)$,
which is the relative boundary in $B_1'$ of the contact set ${B_1'\cap}\{u=0\}$ where the solution saturates the constraint.

\medskip

This problem has been widely studied in the last decades and it has become
a very active field of research after the seminal papers
by Athanasopoulos and Caffarelli~\cite{AtCa04} and Athanasopoulos, Caffarelli and Salsa~\cite{ACS08}.
The key idea introduced in \cite{ACS08} is the use of \textit{Almgren's frequency function}
in the study of both the regularity of the solution $u$ and the properties of the free boundary
$\Gamma(u)$.
This has been the turning point for a long series of results for the constant coefficient case,
{leading to a detailed analysis of the free boundary (see, e.g.,  \cite{ACS08,CSV20,DS15,FoSp16,FS18-1,FS18-1bis,GP09,JP,KPS15} and the references therein).}
The regularity of the solutions of the variable coefficients case has been
considered since the works of Caffarelli~\cite{C79} and Kinderlehrer~\cite{K81}
in the case of smooth coefficients, whereas the problem \eqref{e:minimization} with Sobolev coefficients has been considered by Ural'tseva in \cite{Ural85,Ural86,Ural87}.
The optimal regularity of the solutions and the regularity of a subset of the free boundary (points with selected orders of contact)
have been proven more recently by  Garofalo and Smit Vega Garcia \cite{GS14}, and Garofalo, Petrosyan and Smit Vega Garcia in \cite{GPS16}, respectively, for Lipschitz coefficients
using a generalization of the frequency function {(see also \cite{GPS18} for further results in the Lipschitz setting)}. In case of Sobolev coefficients both topics have been addressed by Koch, R\"uland and Shi~\cite{KRS16,KRS17} by means of Carleman inequalities,
the optimal regularity of solutions is established by R\"uland and Shi~\cite{RS17} for H\"older coefficients, and for a more general notion of quasi-minimizers by Jeon, Petrosyan and Smit Vega Garcia~\cite{JPSm}.
\medskip

In this paper we continue the analysis for quadratic energies with matrix field with Sobolev regularity. In particular, we address the question of the global structure of the free boundary.
In details, we need to consider only the points with finite order of contact: to this aim we write
\[
\Gamma(u):=\Gamma^{\textup{finite}}\cup\Gamma^\infty,
\]
with $\Gamma^{\textup{finite}}$ and $\Gamma^\infty$, to be properly defined in the next sections, representing the points with finite and infinite order of contact.
For $\Gamma^\infty$ no structure at all is expected, in analogy with the case of non zero obstacles studied in \cite{FRRos21,FS22}, and the results on the lack of unique continuation by Pli\v s \cite{Plis}, Miller \cite{Miller}, Filonov \cite{Filonov}, and Mandache \cite{Mandache} for solutions to second order elliptic partial differential equations with H\"older coefficients.
We show that the set of free boundary points with finite order of contact is \textit{rectifiable}, i.e., can be stratified along submanifolds of dimension $n-1$ and class $C^{1}$.

\begin{theorem}\label{t:main}
Let $u\in \mathscr{A}$ be a solution to \eqref{e:minimization} under the hypotheses (H1) and (H2).
The subset of points of the free boundary with finite order of contact $\Gamma^{\textup{finite}}(u)$ is $(n-1)$-rectifiable, i.e., there exists a countable family of $C^1$-submanifolds $M_i\subset B_1'$ of dimension $n-1$ such that
\[
\mathcal{H}^{n-1}\left(\Gamma^{\textup{finite}}(u) \setminus \bigcup_{i\in \N} M_i\right) = 0.
\]
{Furthermore, there exists a set $\Sigma(u)\subset \Gamma^{\textup{finite}}(u)$
with Hausdorff dimension at most $n-2$ such that for every
$x\in\Gamma^{\textup{finite}}(u)\setminus\Sigma(u)$
\[
N_u(x,0^+)\in\{2m,2m-\sfrac12,2m+1\}_{m\in\mathbb N\setminus\{0\}}\,.
\]}
\end{theorem}
{In the statement above, $N_u(x,0^+)$ represents the intrinsic frequency of $u$ at the free boundary point  $x$ (cp. Section~\ref{s.intrinsic frequency} for its definition).}

In addition, if $\mA$ is Lipschitz continuous, a more complete result holds for the whole free boundary $\Gamma(u)$, completely analogous to the case of the classical scalar Signorini problem for the Dirichlet energy as shown in \cite{FS18-1,FS18-1bis}.

\begin{theorem}\label{t:lip}
Let $u\in \mathscr{A}$ be a solution to \eqref{e:minimization} under the hypotheses (H1) with $p=\infty$ and (H2).
Then, the free boundary $\Gamma(u)$ is $(n-1)$-rectifiable with locally finite the Minkowski content: for every $K\subset\subset B_1'$, there exists a constant $C(K)>0$ such that
\[
\mathcal{L}^{n+1}\big(\mathcal{T}_r(\Gamma(u)\cap K)\big) \leq C(K) r^2, \qquad\forall\; r\in (0,1),
\]
where $\mathcal{T}_r(\Gamma(u)\cap K)\subset\R^{n+1}$ is the $r$-tubular neighbourhood of $\Gamma(u)\cap K$.

{Furthermore, there exists a set $\Sigma(u)\subset \Gamma^{\textup{finite}}(u)$
with Hausdorff dimension at most $n-2$ such that for every
$x\in\Gamma^{\textup{finite}}(u)\setminus\Sigma(u)$
\[
N_u(x,0^+)\in\{2m,2m-\sfrac12,2m+1\}_{m\in\mathbb N\setminus\{0\}}\,.
\]}
\end{theorem}
Theorems~\ref{t:main} and \ref{t:lip} are a natural development of a common trend of recent results, which are directed to the understanding the robustness of the techniques conceived for the Laplacian (i.e., costant coefficients operators) in more general
cases (and {more realistic} from the point of view of applications).
In the setting of the thin obstacle problem, we recall the recent contributions on the structure of the regular set for Lipschitz coefficients \cite{GPS16}, for Sobolev coefficients \cite{KRS17} and for H\"older continuous coefficients \cite{JPSm}; for the structure of the singular set with Lipschitz coefficients \cite{GPS18}, and \cite{JPSm} for H\"older coefficients. The rectifiability of the whole free boundary {has been addressed in} \cite{FS20} for the nonlinear case of the area functional.

{One motivation for this study is related to the standard thin obstacle problem, provided the obstacle condition is assigned on a $C^{1,1}$ manifold rather than on a hyperplane. Indeed, a rectification of the manifold leads to a thin obstacle problem as the one stated in \eqref{e:minimization}.
A further motivation is contained in Section~\ref{ss:nonlinear}. Our results allow to
deduce the global structure of the free boundary for solutions to some nonlinear thin
obstacle problems, following the approach used for the area functional in \cite{FS20},
(cf. \cite{DFS22, AAS24} for preliminary results on the regularity of the solutions).}

\subsection{New insights and main difficulties}
The main ideas for this work stem from \cite{FS18-1}. Starting from the groundbreaking papers by Naber and Valtorta~\cite{NaVa1,NaVa2}, it is well-known that a monotone quantity of the type of the energy ratio for harmonic maps can be used to describe the structure of singularities: indeed, if the monotone quantity is able to detect homogeneous blowups at singular points and satisfies a suitable rigidigy property, then general covering and rectifiability arguments lead to the estimate of the measure (actually the Minkowski content) and the rectifiability of the singular set.

This principle has been exploited in \cite{FS18-1} in the case of thin and fractional obstacle problems with constant coefficients, using suitable variants of Almgren's frequency functions, which revelaed itself to be a key tools for this class of problems since \cite{ACS08}.

The main difficulties here are due to extension of such approach to the case of variable coefficients. Indeed, the monotonicity of the frequency is closely related to the linearity of the equations and is valid only for harmonic functions, while in the general case one should find a suitable linear approximation of the local geometry prescribed by the matrix $\mA$.
This is clearly an issue for general nonlinear problems and understanding this question for low regularity matrix fields $\mA$ is a first step towards such program.

We circumvent this difficulty by introducing an intrinsic frequency adapted to the  coefficients matrix $\mA$ ({as opposed to} the natural frequency for variable coefficients, see Section~\ref{s.intrinsic frequency} for more details).
Actually, we need to use three different {forms} of  the frequency, which although different can be suitably compared at the right scales.
In particular, we show a quasi-monotonicity formula for a Dirichlet type frequency for solutions to variable coefficient thin obstacle problems. This idea has been used for the analysis of the classical obstacle problem in \cite{FoGerSp18,AF20}. In the current setting, we couple it with the fundamental insight provided by Simon and Wickramasekera in the framework of $2$-valued minimal graphs, that quasi-monotonicity of the frequency is actually equivalent to a doubling condition on the relevant quantities provided Schauder estimates hold (cf. \cite[Lemma~6.1]{SW16}).

A comment deserves the restriction to Sobolev coefficients as opposite to the more general H\"older setting, for which the analysis of regular and singular points is contained in \cite{JP,JPSm}. In the derivation of the basic estimate on the oscillation of the frequency, as well as for the monotonicity, we differentiate the matrix of coefficients and the gradient of the solutions, and therefore we need enough regularity for $\mA$.

%
%

\section{Preliminaries on the thin obstacle problem}\label{s:preliminari}

Here we recall the hypotheses on the thin obstacle problem we address.
We consider quadratic energies of the form
\[
\int_{B_1^+} \langle\mA(x)\nabla v(x),\nabla v(x)\rangle\, dx,
\]
where the matrix field $\mA$ satisfies the hypotheses:
\begin{itemize}\label{e:A}
 \item[(H1)] $\mA\in W^{1,p}(B_1,\R^{(n+1)\times (n+1)})$, $p\in (n+1,\infty]$ ({in particular, by Morrey's embedding, we have $\mA\in C^{0,\alpha}(B_1,\R^{(n+1)\times (n+1)})$, 
$\alpha:=1-\frac{n+1}p$});
 \item[(H2)]  {$\mA(x)=(a_{ij}(x))_{i,j=1}^{n+1}$ is a symmetric, bounded and coercive matrix, 
i.e. for every $x\in B_1^+$, $i,j\in\{1,\ldots,n+1\}$, and $\xi\in\R^{n+1}$
$a_{i,j}(x)=a_{j,i}(x)$, and}
 \[
 \lambda|\xi|^2\leq  \langle\mA(x)\xi,\xi\rangle\leq\Lambda|\xi|^2\,.
 \]
for some $0<\lambda\leq\Lambda$.
\end{itemize}
Following \cite[Remark 1]{Ural86}, by a means of a change of variables, it is not restrictive to additionally assume that
\begin{itemize}
\item[(H3)] $a_{i,n+1}(x',0)=0$ for all $i=1, \ldots, n$ for every $x\in B_1'$.
\end{itemize}

Under these hypotheses, we can extend all the functions on $B_1$ by even symmetry: with a slight abuse of notation, set
\[
\mA(x',x_{n+1}) = \mA(x',-x_{n+1}), \quad
u(x',x_{n+1})=u(x',-x_{n+1}) \qquad\forall\; x\in B_1^+.
\]
In this way it is equivalent to formulate the problem on $B_1$ with the symmetry condition:
\begin{equation}\label{e:prob}
\min_{v\in \mathcal{A}} \int_{B_1} \langle\mA(x)\nabla v(x),\nabla v(x)\rangle\, dx,
\end{equation}
in the class of functions
\[
\mathcal{A} := \left\{ v\in g+H^1_0(B_1) : v(x',x_{n+1})=v(x',-x_{n+1}),\quad v(x',0)\geq 0\right\},
\]
with $g\in H^{1}(B_1)$ even symmetric and such that  $g\geq 0$ on $B_1'$.

\medskip

By a result due to Ural'tseva \cite[Theorem~3.1]{Ural87} the solution of the thin obstacle problem under assumptions (H1), (H2) and (H3) have H\"older continuous first derivative up to $B_1'$ (for the optimal H\"older exponent see \cite{KRS16,KRS17}).

\begin{theorem}[\cite{Ural87}]\label{t:reg}
For every $g\in H^{1}(B_1)$, even symmetric with $g\geq 0$ on $B_1'$, there exists a unique
solution $u\in \mathcal{A}$ to the thin obstacle problem \eqref{e:prob}.
Moreover, $u \in H_{\textup{loc}}^{2}\cap C_{\textup{loc}}^{1,\beta}(B_1^+\cup B_1')$ for some $\beta\in(0,1)$,
and
\begin{equation}\label{e:C1_alpha reg}
\|u\|_{H^2(B_{\sfrac{1}{2}}^+\cup B_{\sfrac{1}{2}}')} +
\|u\|_{C^{1,\beta}(B_{\sfrac{1}{2}}^+\cup B_{\sfrac{1}{2}}')}
\leq C\, \|u\|_{L^2(B_1)}.
\end{equation}
where $C=C(p,n,\beta,\|\mathbb{A}\|_{W^{1,p}})>0$.
\end{theorem}

\medskip

The Euler--Lagrange equation satisfied by the solution $u$ to the thin obstacle problem are then the following:
\begin{equation}\label{e:ob-pb local}
\begin{cases}
\div\big(\mA(x)\nabla u(x)\big) = 0  & \text{for }\;
x \in B_1 \setminus \big\{(x',0)\,:\, u(x',0) = 0 \big\},\\
\div\big(\mA(x)\nabla u(x)\big) \leq 0 & \text{in the sense of distribution in }\,B_1.
\end{cases}
\end{equation}
Moreover, in view of the assumptions (H2) and (H3), the Signorini complementary condition
in \eqref{e:ob-pb local intro} then reads as
\[
u\, \de_{n+1} u= 0 \qquad \textup{on }\; B_1' .                                                                                                                                                          \]
In the sequel $u$ will always denote a solution to the thin obstacle problem \eqref{e:prob}, unless otherwise stated.

\section{The frequency function}\label{s:frequency}

In this section we introduce a suitable version of Almgren's frequency function.
Let $\phi:[0,\infty) \to [0,\infty)$ be a decreasing $C^{1,1}$ function such that $\phi'(t)<0$ if $\frac{1}{2} < t <1$ and
\[
\phi(t) = 
\begin{cases}
1 & \text{for } \; 0\leq t \leq \frac{1}{2},\\
>0 & \text{for } \; \frac{1}{2} < t <1,\\
0 & \text{for } \; 1 \leq t .\\
\end{cases}
\]
For the sake of simplicity we assume that
\begin{equation}\label{e:struttura phi}
\phi(t) =2(1-t) \qquad \forall\;t\in \left[\frac58,\frac 78\right].
\end{equation}
We define the frequency function of a function $u$ at a point $x_0 \in B_1'$ by
\[
I_u(x_0,r) := \frac{r\, D_u(x_0,r)}{H_u(x_0,r)}
\quad  \forall\; r<1-|x_0|,
\]
where the Dirichlet energy is given by
\[
D_u(x_0,r) := \int \phi\Big(\textstyle{\frac{|x-x_0|}{r}}\Big)\,|\nabla u(x)|^2\d x,
\]
and the ``boundary'' $L^2$ norm of $u$ is given by
\[
H_u(x_0,r) := - \int \phi'\Big(\textstyle{\frac{|x-x_0|}{r}}\Big)
\,\frac{u^2(x)}{|x-x_0|}\,\d x.
\]
Note that the frequency function is well-defined as long as $H_u(x_0,r)>0$.
In what follows, when $u$ is a solution to \eqref{e:prob}, we shall tacitly assume that the latter condition is satisfied when writing $I_u(x_0,r)$.
Indeed, $H_u(x_0,r)>0$ if $x_0\in\Gamma(u)$ by minimality of $u$ and the uniqueness of minimizers.
Otherwise $u\equiv0$ on $B_r(x_0)\setminus B_{\sfrac r2}(x_0)$, in turn implying
$u\equiv0$ on $B_r(x_0)$. This contradicts the fact that $x_0$
is a free boundary point. Analogously, if $x_0\in\Gamma(u)$ then 
$D_u(x_0,r) > 0$. 

Additionally, for later convenience we introduce the following quantities
\[
 G_u(x_0,r):=-r^{-1} \int\phi'\Big({\textstyle{\frac{|x-x_0|}r}\Big)u(x)\nabla u(x)\cdot \frac{x-x_0}{|x-x_0|}}\d x\,,
\]
and
\begin{equation}\label{e:def E}
E_u(x_0,r):= -\int \phi'\Big(\textstyle{\frac{|x-x_0|}{r}}\Big)
\frac{|x-x_0|}{r^2}\Big(\nabla u(x) \cdot\frac{x-x_0}{|x-x_0|}\Big)^2\, \d x.
\end{equation}
In particular, note that $E_u(x_0,r)H_u(x_0,r)-G_u^2(x_0,r)\geq 0$ by Cauchy-Schwartz inequality.

Finally, for every $x_0 \in \Gamma(u)$ and $r>0$, the rescalings
of a solution $u$ are given by
\begin{gather}\label{e:rescaling-1}
u_{x_0,r}(y) :=
\frac{r^{\sfrac n2}}{H_u^{\sfrac12}(x_0,r)}u(x_0+r\,y)
\quad\quad \forall\; y\in 
B_{\frac{1-|x_0|}{r}},
\end{gather}
so that
\[
H_{u_{x_0,r}}(\underline{0},1)=1.
\]
We shall always omit to write the base point 
$x_0$ in the notation of $I_u$, $D_u$, $H_u$, $E_u$, $G_u$ when $x_0 = \underline{0}$.
\medskip

By a simple corollary of Theorem \ref{t:reg} we have the following compactness: if $(u_j)_{j\in \N}$ are such that
\[
\sup_{j\in \N} \big(D_{u_j}(1) + H_{u_j}(1)\big) < + \infty,
\]
then $u_j$ is uniformly bounded in $H^{1}(B_s)$ for every $s<1$.
Therefore, if $u_j$ are solutions to the thin obstacle problem \eqref{e:prob} satisfying the hypotheses (H1)-(H3) ({holding uniformly in $j$ for varying matrix fields $\mA_j$, } with the same constants $p, \lambda, \Lambda$), then by Theorem \ref{t:reg} there exists a function $u\in C^{1,\beta}(B_1^+)$ and a subsequence $(u_{j'})\subset (u_j)$ such that
\[
u_{j'} \to u\qquad \textup{in}\quad C^{1,\gamma}_{\loc}(B_1^+\cup B_1')\quad\forall\;\gamma\in(0,\beta),
\]
where $\beta$ is the constant in Theorem \ref{t:reg}.

In particular, the following compactness result holds for the rescaling of solutions.
\begin{proposition}\label{c:compactness}
Let $(u_j)_{j\in\N}$ be a sequence of solutions to the thin obstacle problem \eqref{e:prob}
{for varying matrix fields $\mA_j$ under assumptions (H1)-(H3) holding uniformly in $j$}, with $x_j\in \Gamma(u_j)\cap B_{\sfrac12}'$ for every $j \in \N$.
Assume that
$$\sup_{j}I_{u_j}(x_j,\varrho_j)<\infty \qquad \textup{ for some $\varrho_j\downarrow 0$.}$$ Then, there exists a subsequence
$(j')$ such that $x_{j'}\to x_\infty\in\bar{B}_{\sfrac12}'$ and
a function $v_\infty$ such that 
$v_{j'}:=(u_{j'})_{x_{j'},\varrho_{j'}}$ satisfy as $j'\to \infty$
\begin{gather}
v_{j'}\to v_\infty \quad
\text{in}\quad C^{1,\gamma}_{\loc}(B_1^+\cup B_1') \quad\forall\;\gamma\in(0,\beta),
\label{e:cpt2}\\
\underline{0}\in \{x\in\bar{B}_{\sfrac12}:\,v_\infty(x)=|\nabla v_\infty(x)|=0\}\,.\label{e:cpt3}
\end{gather}
Moreover, $v_\infty$ is a solution the thin obstacle problem for the constant coefficients quadratic energy having density $\xi\mapsto \langle\mA(x_\infty)\xi,\xi\rangle$.
\end{proposition}

\begin{proof}
From the assumption on the frequency, clearly the functions $v_{j}:=(u_{j})_{x_{j},\varrho_{j}}$ satisfy
\[
\sup_{j\in \N} \big(D_{v_j}(1) + H_{v_j}(1)\big) < + \infty.
\]
Therefore, by compactness there exists a subsequence $(j')$ such that the points $x_{j'}$ converge to some $x_\infty\in \bar B_{\sfrac12}'$, and $v_{j'}$ converge to a limiting function $v_\infty$ in the sense of \eqref{e:cpt2}.
Since $v_{j'}(\underline{0})=|\nabla v_{j'}(\underline{0})|=0$ (because $x_{j'}\in \Gamma(u_{j'})$),
by the convergence \eqref{e:cpt2} we infer also \eqref{e:cpt3}.

Finally, the fact that $v_\infty$ is a solution to a thin obstacle problem with fixed coefficients $\mA(x_\infty)$ follows either by a $\Gamma$-convergence result or by passing into the limit in the weak formulation of the Euler--Lagrange equations \eqref{e:ob-pb local} characterizing the solutions, thanks to the convergence \eqref{e:cpt2} and the continuity of the matrix field $\mA$.
\end{proof}

\subsection{Doubling estimates}

To establish quasi-monotonicity of the frequency at distinguished points of the free boundary, 
we follow a perturbative approach developed for the classical obstacle in \cite{FoGerSp18}.
This approach is coupled with a fundamental insight present in the work of Simon and Wickramasekera~\cite{SW16} on minimal immersion, who highlighted that the doubling condition is equivalent to the quasi-monotonicity of the frequency.

More in details, we shall consider $x_0\in \Gamma(u)$ satisfying the following hypothesis:
\begin{itemize}
\item[(H4)] $x_0\in\Gamma(u)\cap B_{\sfrac12}'$ such that
\begin{equation}\label{e:limsup frequency}
\mathfrak{m}(x_0):=\sup_{r\in(0,\sfrac12)}I_u(x_0,r)<\infty
\end{equation}
\end{itemize}
As shown in Appendix~\ref{a:order of contact} condition \eqref{e:limsup frequency} 
is equivalent to assume that $u$ has a finite order of contact with the null obstacle at 
$x_0\in\Gamma(u)$ {(cf. the Introduction).}

A first step towards the monotonicity of the frequency is to establish a lower bound.
A more refined version of the ensuing result will follow after showing 
the quasi-monotonicity of the frequency (cf. \eqref{e:freq lb bis}).
All the constants that will appear in the results below can depend on the parameters $p, \lambda, \Lambda$ of (H1)-(H3), even though it will never be explicitly highlighted.

\begin{lemma}\label{l:freq lb}
For every $\mathfrak{m}_0>0$ there exist constants $\varrho, C>0$ depending on
$\mathfrak{m}_0$ with this property. If $u$ is a solution to \eqref{e:prob} under assumptions (H1)-(H3), for every $x_0 \in \Gamma(u)\cap B_{\sfrac12}'$ satisfying (H4) with $\mathfrak{m}(x_0)\leq \mathfrak{m}_0$, then
\begin{equation}\label{e:freq lb}
I_u(x_0,r)\geq C\qquad \forall\; r\in (0,\varrho]. 
\end{equation}
\end{lemma}

\begin{proof}
Assume by contradiction that there exist solutions $u_j$ and points $x_j\in \Gamma(u_j)\cap B'_{\sfrac12}$ with $\mathfrak{m}(x_j)\leq \mathfrak{m}_0$ for some $\mathfrak{m}_0<\infty$
as in the statement, such that for a suitable choice of radii $r_j\downarrow 0$ we have that
\[
I_{u_j}(x_j,r_j)\leq \sfrac1j.
\]
For $j$ sufficiently large $I_{u_j}(x_j,2r_j)\leq \mathfrak{m}_0$, and therefore by
Corollary~\ref{c:compactness}, up to a subsequence not relabeled and keeping
the notation introduced there, we conclude that the functions $v_j:=(u_j)_{x_j,2r_j}$
converges in $C^1(B_1)$ to some function $v_\infty$ that minimizes
\[
v\mapsto \int_{B_1}\langle \mA(x_\infty)\nabla v,\nabla v\rangle dx,
\]
among all functions $v\in v_\infty+H^{1}_0(B_1)$ with $v(x',0)\geq 0$ on $B_1'$.
Moreover, by strong convergence, we infer that $H_{v_\infty}(1)=H_{(u_j)_{x_j,2r_j}}(1)=1$
and
\[
I_{v_\infty}(\sfrac12)=\lim_jI_{(u_j)_{x_j,2r_j}}(\sfrac12)=\lim_jI_{u_j}(x_j,r_j)=0.
\]
Therefore, $D_{v_\infty}(\sfrac12)=0$ and thus $v_\infty\equiv 0$ on $B_{\sfrac12}$, 
a contradiction being $v_\infty$ analytic on $B_1\setminus B_1'$.
\end{proof}

Next we prove the above mentioned doubling estimates.

\begin{proposition}\label{p:doubling}
{For every $\mathfrak{m}_0>0$ there exist constants $\varrho, C>0$ depending on
$\mathfrak{m}_0$ with this property. If $u$ is a solution to \eqref{e:prob} under assumptions
(H1)-(H3), for every $x_0 \in \Gamma(u)\cap B_{\sfrac12}'$ satisfying (H4) with $\mathfrak{m}(x_0)\leq \mathfrak{m}_0$, then}
\begin{equation}\label{e:H doubling}
C^{-1}\leq\frac{H_u(x_0,2r)}{H_u(x_0,r)}\leq C,  
\qquad 1\leq\frac{D_u(x_0,2r)}{D_u(x_0,r)}\leq C \qquad \forall\;r\in (0,\varrho].
\end{equation}
\end{proposition}

\begin{proof}
First note that by the very definition $D_u(x_0,r)\leq D_u(x_0,2r)$. 
Assume by contradiction that there exist solutions $u_j$, points
$x_j\in \Gamma(u_j)\cap B'_{\sfrac12}$ with $\mathfrak{m}(x_j)\leq \mathfrak{m}_0$ for some $\mathfrak{m}_0<\infty$
as in the statement and radii $r_j\downarrow 0$ such that 
\[
\lim_j\frac{H_{u_j}(x_j,2r_j)}{H_{u_j}(x_j,r_j)}\in\{0,\infty\},
\quad\text{and/or}\quad
\lim_j\frac{D_{u_j}(x_j,2r_j)}{D_{u_j}(x_j,r_j)}=\infty.
\]
For $j$ sufficiently large $I_{u_j}(x_j,2r_j)\leq \mathfrak{m}_0$,  
thus by Corollary~\ref{c:compactness}, up to a subsequence not relabeled, $v_j:=(u_j)_{x_j,2r_j}$ 
converges in $C^1(B_1)$ to some function $v_\infty$ that minimizes
\[
v\mapsto \int_{B_1}\langle\mA(x_\infty)\nabla v,\nabla v\rangle dx,
\]
among all functions $v\in v_\infty+H^{1}_0(B_1)$ with $v(x',0)\geq 0$ on $B_1'$.
In particular, we conclude that 
\[
\lim_{j\to\infty}\frac{H_{u_j}(x_j,2r_j)}{H_{u_j}(x_j,r_j)}=\frac{H_{v_\infty}(1)}{H_{v_\infty}(\sfrac12)}\in [C^{-1}, C],
\]
and 
\[
\lim_{j\to\infty}\frac{D_{u_j}(x_j,2r_j)}{D_{u_j}(x_j,r_j)}=\frac{D_{v_\infty}(1)}{D_{v_\infty}(\sfrac12)}\in [1,C],
\]
for some constant $C>0$ (depending on $\mathfrak{m}_0$) because of the doubling estimates satisfied by $v_\infty$ which is a solution to an obstacle problem with constant coefficients and frequency $I_{v_\infty}(1)$ bounded by $\mathfrak{m}_0$ (e.g., cf. \cite[Corollary~2.8]{FS18-1}). This gives the desired contradiction.
\end{proof}

In a similar fashion, the frequency computed at nearby points can be compared at radii that are bigger than the distance between the points.

\begin{lemma}\label{l:lim uniforme}
{For every $\mathfrak{m}_0>0$ there exist constants $\varrho, C>0$ depending on
$\mathfrak{m}_0$ with this property. If $u$ is a solution to \eqref{e:prob} under assumptions (H1)-(H3), 
for every $x_0 \in \Gamma(u)\cap B_{\sfrac12}'$ satisfying (H4) with $\mathfrak{m}(x_0)\leq \mathfrak{m}_0$, then}
for all $r\in(0,\varrho]$, $x\in B'_{\sfrac r2}(x_0)$, and $t\in[r,\varrho]$
\begin{gather}\label{e:H limitato}
\frac{H_u(x,t)}{H_u(x_0,t)},\quad
\frac{D_u(x,t)}{D_u(x_0,t)} \in [C^{-1},C].
\end{gather}
In particular, the frequency function of $u$ is well-defined at every point $x\in B_{\sfrac r2}'(x_0)$ at the scales $t\in[r,\varrho]$  and
\begin{equation}\label{e:I limitato}
\big| I_u(x_0,t) - I_u(x,t) \big|\leq C.
\end{equation}
\end{lemma}

\begin{remark}
We stress that the conclusions of Lemma \ref{l:lim uniforme} hold even for points $x$ not necessarily in the free boundary.
\end{remark}

\begin{proof}
The proof proceeds analogously to the previous ones  by contradiction: assume there exist functions $u_j$, points $z_j\in \Gamma(u_j)\cap B_{\sfrac12}'$ with $\mathfrak{m}(z_j)\leq \mathfrak{m}_0 <\infty$ and $x_j\in B'_{\sfrac{r_j}2}(z_j)$ contradicting one of the two sets of inequalities in \eqref{e:H limitato} for some sequence
$t_j\in[r_j,\varrho_j]$ with $\varrho_j\downarrow0$.
As above, we can apply Corollary~\ref{c:compactness}, thus
(up to passing to a subsequence not relabeled) there exists $v_\infty$ 
such that $v_j:=(u_j)_{z_j,t_j}\to v_\infty$ in $C^{1,\beta}_{\textup{loc}}(B_2)$ with $v_\infty$ solution to the thin obstacle problem {with matrix fields $\mA(x_\infty)$}.
Clearly, we may also assume that $t_j^{-1}(x_j-z_j)\to y_\infty\in \bar B'_{\sfrac12}$ (note that $\sfrac{r_j}{t_j}\leq 1$).

Assume now that the first set of inequalities in \eqref{e:H limitato} is contradicted, {\ie}
\[
\lim_{j}\frac{H_{u_j}(x_j,t_j)}{H_{u_j}(z_j,t_j)} \in \{0,\infty\}.
\]
By the convergence of $v_j$ to $v_\infty$ we then deduce that 
\[
H_{v_\infty}(y_\infty,1)=\lim_jH_{v_j}(t_j^{-1}(x_j-z_j),1)
=\lim_j\frac{H_{u_j}(x_j,t_j)}{H_{u_j}(z_j,t_j)}\,. 
\]
Thus $H_{v_\infty}(y_\infty,1)\in \{0,\infty\} \cap \R = \{0\}$.
Given that $v_\infty$ is analytical in $B_2\setminus \{x_{n+1} =0\}$,
by unique continuation we conclude that $v_\infty\equiv 0$ in $B_2$, against 
the assumption $H_{v_\infty}(1) = \lim_j H_{v_j}(1) =1$.

In case the second set of inequalities in \eqref{e:H limitato} is contradicted, {\ie}
\[
\lim_{j}\frac{D_{u_j}(x_j,t_j)}{D_{u_j}(z_j,t_j)} \in \{0,\infty\}.
\]
we have on one hand that $D_{v_\infty}(1)=\lim_jD_{v_j}(1)=\lim_jI_{u_j}(z_j,t_j)\in[C,\mathfrak{m}_0]$, where $C$ is the constant of the lower bound found in Lemma \ref{l:freq lb}
and $\mathfrak{m}_0$ is the upper bound for the $\mathfrak{m}(z_j)$.
On the other hand we have that 
\[
D_{v_\infty}(y_\infty,1)=\lim_jD_{v_j}(t_j^{-1}(x_j-z_j),1)
=\lim_jt_j\frac{D_{u_j}(x_j,t_j)}{H_{u_j}(z_j,t_j)} 
=\lim_jI_{u_j}(z_j,t_j)\frac{D_{u_j}(x_j,t_j)}{D_{u_j}(z_j,t_j)}\,.
\]
Therefore, under the contradiction assumption we infer that $D_{v_\infty}(y_\infty,1)\in \{0, \infty\}\cap\R=\{0\}$ and taking into account the analiticity of the solutions the last 
equality implies $v_\infty\equiv 0$, which is a contradiction.

Finally, as a byproduct of the first estimate in \eqref{e:H limitato} 
the frequency function is well-defined for $t\in[r,\varrho]$, provided that 
$x\in B_{\sfrac r2}'(x_0)$.
Moreover, \eqref{e:I limitato} follows straightforwardly from \eqref{e:H limitato}:
\begin{align*}
\Big\vert I_u(x_0,t) - I_u(x,t) \Big\vert & = \Big\vert I_u(x_0,t)
\left(1-\frac{D_u(x,t)}{D_u(x_0,t)}\cdot\frac{H_u(x_0,t)}{H_u(x,t)} \right)  
\Big\vert\leq C. \qedhere
\end{align*}
\end{proof}

As an immediate consequence, the doubling estimates hold not only for the points on the free boundary, but also for nearby points at suitable scales. {This information will be crucial to bound error terms in the almost monotonicity formulas in the sequel (cf.  \eqref{e:epsilon D}, \eqref{e:e_Dprime}).} 

\begin{corollary}\label{c:freq ben def}
{For every $\mathfrak{m}_0>0$ there exist constants $\varrho, C>0$ depending on
$\mathfrak{m}_0$ with this property. If $u$ is a solution to \eqref{e:prob} under assumptions (H1)-(H3), 
for every $x_0 \in \Gamma(u)\cap B_{\sfrac12}'$ satisfying (H4) with $\mathfrak{m}(x_0)\leq \mathfrak{m}_0$, then} for 
every $r\in(0,\varrho]$, $x\in B'_{\sfrac r2}(x_0)$, and $t\in[r,\varrho]$
\begin{equation}\label{e:H doublingbis}
\frac{H_u(x,2t)}{H_u(x,t)},\;
\frac{D_u(x,2t)}{D_u(x,t)}\leq [C^{-1},C].
\end{equation}
\end{corollary}
\begin{proof}
The proof follows straightforwardly from Proposition \ref{p:doubling} and 
Lemma \ref{l:lim uniforme} once the constants 
are chosen in such a way to apply the above mentioned results.
\end{proof}

\subsection{Almost monotonicity of the frequency}

By means of the doubling estimates established in Proposition~\ref{p:doubling}, and arguing similarly to \cite[Theorem~2.2]{FoGerSp18}, we can conclude almost monotonicity of the frequency $I_u(x_0,\cdot)$ under the condition
$\mathbb{A}(x_0)=\textup{Id}$.
In fact, we establish quasi-monotonicity for all points $x_0$ and radii $r\geq (\mathfrak{a}(x_0))^{\frac{1}{\alpha}}$, recall that $\alpha=1-\frac{n+1}{p}$, provided that
\[
\mathfrak{a}(x_0):=|\mA(x_0)-\textup{Id}|.
\]
Notice that $\mathfrak{a}(x_0)=0$ if and only if $\mA(x_0)=\textup{Id}$.
We start off proving a couple of preliminary results. 
In what follows, differentiation with respect to the radius shall be denoted by a prime and we write
\begin{align}\label{e:variationG}
\epsilon_D(x,t):= G_u(x,t)- D_u(x,t)\,,
\end{align}
\begin{align}\label{e:variationDprime}
\epsilon_{D'}(x,t):= tD_u'(x,t) - (n-1)D_u(x,t)- 2tE_u(x,t)\,.
\end{align}

\begin{lemma}\label{l:stima J1}
{For every $\mathfrak{m}_0>0$ there exist constants $\varrho, C>0$ depending on
$\mathfrak{m}_0$ with this property. If $u$ is a solution to \eqref{e:prob} under assumptions (H1)-(H3), 
for every $x_0 \in \Gamma(u)\cap B'_{\sfrac12}$ satisfying (H4) with $\mathfrak{m}(x_0)\leq \mathfrak{m}_0$, then} for 
every $r\in (0,\varrho]$, $x\in B'_{\sfrac r2}(x_0)$, and $t\in [r,\varrho]$
\begin{align}\label{e:variationD errore}
|\epsilon_D(x,t)|\leq C &\big([\mA]_{0,\alpha}(|x-x_0|+t)^\alpha+\mathfrak{a}(x_0)\big)\notag\\
&\qquad\Big(D_u(x,t)+t^{-\sfrac12}H_u^{\sfrac12}(x,t)D_u^{\sfrac12}(x,t)\Big)\,,
\end{align}
and
\begin{align}\label{e:variationD prime errore}
|\epsilon_{D'}(x,t)|\leq C\big([\mA]_{0,\alpha}(|x-x_0|+t)^\alpha+\mathfrak{a}(x_0)\big) D_u(x,t)\,.
 \end{align}
\end{lemma}

\begin{proof}
Without loss of generality we suppose $x_0=\underline{0}$.
Fix a point $x\in B_{\sfrac12}'$, and set $\mB(z):=\mA(z)-\textrm{Id}$ {for every $z\in B_1'$}.
To infer \eqref{e:variationD errore} and \eqref{e:variationD prime errore} we consider the equation satisfied by $u$ and test it
with a suitable function. 

We start noticing that
\begin{align}\label{e:e_D}
\epsilon_D(x,t) & = G_u(x,t)-D_u(x,t)\notag\\
&=-\int\nabla u(z)\cdot \nabla\Big(u(z)\phi\Big({\textstyle{\frac{|z-x|}t}}\Big)\Big)\d z\notag\\
&=\int\mB(z)\nabla u(z)\cdot \nabla\Big(u(z)\phi\Big({\textstyle{\frac{|z-x|}t}}\Big)\Big)\d z,
\end{align}
where in the last equality we use that $u$ is a solution to \eqref{e:ob-pb local} (tested with $u(z)\phi\big({\textstyle{\frac{|z-x|}t}}\big)$) and Signorini ambiguous boundary conditions.
In order to prove \eqref{e:variationD errore}, we use \eqref{e:e_D}: for $\varrho$ sufficiently small, by the H\"older inequality we get
\begin{align}\label{e:epsilon D}
|\epsilon_D(x,t)|&\leq\big( [\mA]_{0,\alpha}(|x|+t)^\alpha+\mathfrak{a}(x_0)\big)\notag\\
&\qquad\Big(D_u(x,t)-t^{-1} \int_{B_t(x)\setminus B_{\sfrac t2}(x)}
\phi'\Big({\textstyle{\frac{|z-x|}t}}\Big)|u(z)||\nabla u(z)|\d z\Big)\notag\\
&\leq C \big( [\mA]_{0,\alpha}(|x|+t)^\alpha+\mathfrak{a}(x_0)\big)
\Big(D_u(x,t)+t^{-\sfrac12}H_u^{\sfrac12}(x,t)D_u^{\sfrac12}(x,2t)\Big) \,,
 \end{align}
with $C=C(\|\phi'\|_\infty)$.
We conclude the estimate in \eqref{e:variationD errore} in view of Corollary~\ref{c:freq ben def} 
(cf. \eqref{e:H doublingbis}).

To prove \eqref{e:variationD prime errore} we argue similarly, and test the equation with
the function $w$ defined as the even extension across $B_1'$ of the restriction to $B_1^+$ 
of $\phi\big({\textstyle{\frac{|z-x|}t}}\big)\nabla u(z)\cdot (z-x)$. Note that $w$ is an 
admissible test in view of the $H^2_{\loc}(B_1^\pm\cup B_1')$ regularity of $u$. Additionally, (H3) and Signorini's ambiguous boundary conditions imply that $(\mA(z)\nabla u(z)\cdot e_{n+1})\nabla u(z)\cdot (z-x)={a_{n+1,n+1}(z)\partial_{n+1}u(z)=}0$ on $B_1'$.
In view of this we compute explicitly  using the divergence theorem
\begin{align}\label{e:e_Dprime}
 \int&\mB(z)\nabla u(z)\cdot \nabla w\d z=\int\nabla u(z) \cdot \nabla w\,\d z\notag\\
 &=D_u(x,t)+\frac12\int\phi\Big({\textstyle{\frac{|z-x|}t}}\Big)\nabla(|\nabla u(z)|^2)\cdot (z-x)\,\d z \notag\\
&\qquad +\frac1t\int\phi'\Big({\textstyle{\frac{|z-x|}t}}\Big){\textstyle{\frac{(\nabla u(z)\cdot (z-x))^2 }{|z-x|}}}\,\d z\notag\\
 &=-\frac{n-1}{2}D_u(x,t)+\frac1{t}\int\phi'\Big({\textstyle{\frac{|z-x|}t}}\Big)
 \Big(-{\textstyle{\frac{|\nabla u(z)|^2}2|z-x|+\frac{(\nabla u(z)\cdot (z-x))^2 }{|z-x|}}}\Big)\d z\notag\\
 &=-\frac{n-1}{2}D_u(x,t)+\frac t2D_u'(x,t)-tE_u(x,t)
= -\epsilon_{D'}(x,t).
\end{align}
To establish \eqref{e:variationD prime errore} we can then proceed similarly as above:
\begin{align*}
 |\epsilon_{D'}(x,t)|\leq&  C\big([\mA]_{0,\alpha}(|x|+t)^\alpha+\mathfrak{a}(x_0)\big)\cdot\Big(D_u(x,t)+\\
&\quad + t\int\phi\big({\textstyle{\frac{|z-x|}t}}\big)|\nabla^2 u(z)||\nabla u(z)|\,\d z
 +\int_{B_t(x)\setminus B_{\sfrac t2}(x)}|\nabla u(z)|^2\d z\Big)\notag\\
 &\leq  C \big([\mA]_{0,\alpha}(|x|+t)^\alpha+\mathfrak{a}(x_0)\big)\cdot\\
&\qquad \cdot\big(D_u(x,t)+D_u^{\sfrac12}(x,t)D_u^{\sfrac12}(x,2t)+D_u(x,2t)\big),
\end{align*}
where we used the $H^2_{\loc}(B_1^\pm)$ regularity estimates of $u$.
The conclusion then follows from the doubling properties 
of $D_u(x,\cdot)$ (cf. Corollary~\ref{c:freq ben def}).
\end{proof}
We establish next a similar result for $H_u(x_0,\cdot)$ together with a quasi-monotonicity formula.

\begin{lemma}\label{l:monotonia H}
{For every $\mathfrak{m}_0>0$ there exist constants $\varrho, C>0$ depending on
$\mathfrak{m}_0$ and $[\mA]_{0,\alpha}$ with this property. If $u$ is a solution to \eqref{e:prob} under assumptions (H1)-(H3), for every $x_0 \in \Gamma(u)\cap B_{\sfrac12}'$ satisfying (H4) with $\mathfrak{m}(x_0)\leq \mathfrak{m}_0$,} if
$(\mathfrak{a}(x_0))^{\sfrac1\alpha}\leq \varrho$, then
\begin{equation}\label{e:monotonia H}
((\mathfrak{a}(x_0))^{\sfrac{1}{\alpha}},\varrho]\ni t\mapsto
\frac{H_u(x_0,t)}{t^n}\cdot\exp\big(C t^\alpha\big)
\quad\text{is nondecreasing,}
\end{equation}
and
\begin{equation}\label{e:monotonia H2}
((\mathfrak{a}(x_0))^{\sfrac{1}{\alpha}},\varrho]\ni t\mapsto
\frac{H_u(x_0,t)}{t^{n+2\mathfrak{m}_0}}\cdot\exp\big(-C t^\alpha\big)
\quad\text{is nonincreasing.}
\end{equation}
In particular, for all $(\mathfrak{a}(x_0))^{\sfrac{1}{\alpha}}\leq r\leq s\leq \varrho$
\begin{equation}\label{e:L2 vs H}
\int_{B_s(x_0)\setminus B_r(x_0)}|u(x)|^2\d x\leq C\,s H_u(x_0,s)\,.
\end{equation}
\end{lemma}

\begin{proof}
First note that by scaling and a direct differentiation we easily get \begin{align}\label{e:variationHprime}
H_u'(x,t)&=\frac{n}{t}\,H_u(x,t)
 -2\,t^{-1} \int\phi'\Big({\textstyle{\frac{|z-x|}t}\Big)u(z)\nabla u(z)\cdot \frac{z-x}{|z-x|}}\d z\notag\\
&=\frac{n}{t}\,H_u(x,t)+2G_u(x,t).
\end{align}
We employ equalities \eqref{e:variationG}, \eqref{e:variationHprime} and estimate \eqref{e:variationD errore} (with $x=x_0$)
to deduce that
\begin{align*}
\Big|\frac{\d}{\d t}&\Big(\ln\big({\textstyle{\frac{H_u(x_0,t)}{t^n}}}\big)+2\int_t^{\varrho}{\textstyle{\frac{I_u(x_0,s)}s}}\d s\Big)\Big|= \frac{2|\epsilon_D(x_0,t)|}{H_u(x_0,t)}\\
&\leq  C \big([\mA]_{0,\alpha}t^{\alpha}+\mathfrak{a}(x_0) \big) {t^{-1}}\big(I_u(x_0,t)+
I_u^{\sfrac12}(x_0,t)\big)
\leq C t^{\alpha-1}\,,
\end{align*}
where $C=C([\mA]_{0,\alpha}, \mathfrak{m}_0)>0$ and we have used that
$(\mathfrak{a}(x_0))^{\sfrac1\alpha}<t<1$. The conclusion in
\eqref{e:monotonia H} then follows at once by direct integration.
Similarly, using $I_u(x_0,s)\leq \mathfrak{m}_0$, we have
\begin{align*}
\frac{\d}{\d t}\ln\big({\textstyle{\frac{H_u(x_0,t)}{t^{n+2\mathfrak{m}_0}}}}\big)& \geq \frac{\d}{\d t}\Big(\ln\big({\textstyle{\frac{H_u(x_0,t)}{t^n}}}\big)+2\int_t^{\varrho}{\textstyle{\frac{I_u(x_0,s)}s}}\d s\Big)\geq -C t^{\alpha-1}\,,
\end{align*}
and \eqref{e:monotonia H2} follows.
 
Finally, the proof of \eqref{e:L2 vs H} is a simple consequence of Fubini theorem
by taking advantage of \eqref{e:monotonia H}.
\end{proof}

Thanks to Lemmata~\ref{l:stima J1} and \ref{l:monotonia H} we are now ready to establish the quasi-monotonicity of the frequency function at free boundary points $x_0$ with $\mathfrak{a}(x_0)=0$.

\begin{proposition}\label{p:monotonia freq}
For every $\mathfrak{m}_0>0$ there exist constants $\varrho, C>0$ depending on
$\mathfrak{m}_0$, $[\mA]_{0,\alpha}$ and $\alpha$ with this property. If $u$ is a solution to \eqref{e:prob} under assumptions (H1)-(H3), for every $x_0 \in \Gamma(u)\cap B_{\sfrac12}'$ satisfying (H4) with $\mathfrak{m}(x_0)\leq \mathfrak{m}_0$, if
$\max\{(\mathfrak{a}(x_0))^{\sfrac1\alpha},r\}\leq \varrho$ and $x\in B'_{\sfrac r2}(x_0)$, then
\begin{equation}\label{e:Iuprime}
I_u'(x,t)=\frac{2t}{H_u^2(x,t)}\big(E_u(x,t)H_u(x,t)-G_u^2(x,t)\big)+R_u(x,t)\,,
\end{equation}
and
\begin{equation}\label{e:stimaRut}
 |R_u(x,t)|\leq C\,t^{\alpha-1}I_u(x,t)\qquad
 \forall\;t\in \big[\max\{(\mathfrak{a}(x_0))^{\sfrac1\alpha}, r\},\varrho\big]\,.
\end{equation}
In particular, if $\mathfrak{a}(x_0)=0$, then for every $x\in B'_{\sfrac r2}(x_0)$, $r<\varrho$,
\begin{equation}\label{e:almost monotonicity}
[r,\varrho]\ni t\mapsto e^{C t^{\alpha}}I_u(x,t)
\quad\text{is non-decreasing}.
\end{equation}
\end{proposition}

\begin{proof}
Without loss of generality we prove the result in case $x_0=\underline{0}$. 
We take advantage of formulas \eqref{e:variationG}-\eqref{e:variationD prime errore} in Lemma~\ref{l:stima J1} and \eqref{e:variationHprime} in Lemma~\ref{l:monotonia H} to deduce that
\begin{align}\label{e:Iprimeparziale}
 \frac{I_u'(x,t)}{I_u(x,t)}&=\frac 1t+\frac{D_u'(x,t)}{D_u(x,t)}-\frac{H_u'(x,t)}{H_u(x,t)}
\stackrel{\eqref{e:variationDprime}}{=}2\,\frac{E_u(x,t)}{D_u(x,t)}-2\frac{G_u(x,t)}{H_u(x,t)} +\frac{\epsilon_{D'}(x,t)}{t D_u(x,t)} \notag\\
&=\frac{2}{D_u(x,t)H_u(x,t)} \big(E_u(x,t)H_u(x,t)-G_u(x,t)D_u(x,t)\big)+\frac{\epsilon_{D'}(x,t)}{t D_u(x,t)}
\notag\\
&\stackrel{\eqref{e:variationG}}{=}\frac{2}{D_u(x,t)H_u(x,t)}\big(E_u(x,t)H_u(x,t)-G_u^2(x,t)\big)\notag\\ 
&+\,\frac{\epsilon_{D'}(x,t)}{t D_u(x,t)}+2\,\frac{\epsilon_D(x,t)G_u(x,t)}{D_u(x,t)H_u(x,t)}.
\end{align}
Thus, for $t\in(0,\sfrac12)$ we conclude equality \eqref{e:Iuprime}, i.e.
\begin{align}\label{e:freq derivative}
 I_u'(x,t)&-\frac{2t}{H_u^2(x,t)}\big(E_u(x,t)H_u(x,t)-G_u^2(x,t)\big)=
 \frac{\epsilon_{D'}(x,t)}{H_u(x,t)}+2t\frac{G_u(x,t)}{H_u^2(x,t)}\epsilon_D(x,t)\notag\\
&\stackrel{\eqref{e:variationG}}{=} \frac{\epsilon_{D'}(x,t)}{H_u(x,t)}+2t\frac{\epsilon_D^2(x,t)}{H_u^2(x,t)}+2I(x,t)\frac{\epsilon_D(x,t)}{H_u(x,t)}=:R_u(x,t).
 \end{align}
 To estimate $R_u$ we note that, if $t\geq \max\{ (\mathfrak{a}(x_0))^{\sfrac1\alpha},r\}$ and 
{$x\in B'_{\sfrac r2}(x_0)$}, then
\begin{align*}
\Big|\frac{\epsilon_{D'}(x,t)}{H_u(x,t)}\Big|&\stackrel{\eqref{e:variationD prime errore}}{\leq}
C t^{\alpha-1} I_u(x,t)\,,
\end{align*}
in turn implying that 
\begin{align*}
\frac{t\epsilon_D^2(x,t)}{H_u^2(x,t)}+ I_u(x,t)\frac{|\epsilon_D(x,t)|}{H_u(x,t)}
 \stackrel{\eqref{e:variationD errore}}{\leq}
Ct^{\alpha-1} I_u(x,t),
\end{align*}
with $C=C(\mathfrak{m}(x_0),[\mA]_{0,\alpha})$, where we use the {local} uniform upper bound 
on $I_u(x, t)$ given by Lemma \ref{l:lim uniforme}. Therefore, estimate \eqref{e:stimaRut} follows straightforwardly.

In particular, if $\mathfrak{a}(x_0)=0$, then \eqref{e:stimaRut}
holds true for all $t\in[r,\varrho]$, and thus \eqref{e:almost monotonicity} follows by direct
integration of \eqref{e:Iuprime} by taking into account estimate \eqref{e:stimaRut} and 
Cauchy-Schwarz inequality that implies $E_u(x,t)H_u(x,t)-G_u^2(x,t)\geq 0$.
\end{proof}

\begin{remark}\label{r:monotonia freq}
It is also convenient to highlight a different expression of the derivative of 
the frequency function for later purposes: from \eqref{e:freq derivative} we get
\begin{equation}\label{e:Iuprime bis}
I_u'(x,t)=\frac{2t}{H_u^2(x,t)}\big(E_u(x,t)H_u(x,t)-D_u^2(x,t)\big)+\widetilde{R}_u(x,t)\,,
\end{equation}
where
\begin{equation}\label{e.Rtilde}
 \widetilde{R}_u(x,t):=\frac{\epsilon_{D'}(x,t)}{H_u(x,t)}-
 2I_u(x,t)\frac{\epsilon_D(x,t)}{H_u(x,t)}\,.
\end{equation}
In particular, if (H4) is satisfied in $x_0$, {$x\in B'_{\sfrac r2}(x_0)$} and $t\geq\max\{(\mathfrak{a}(x_0))^{\sfrac1\alpha},r\}$, then
\begin{equation}\label{e:stimaRut bis}
 |\widetilde{R}_u(x,t)|\leq C\,t^{\alpha-1}I_u(x,t)\,.
\end{equation}
\end{remark}

An additive quasi-monotonicity formula is then easily deduced. 

\begin{corollary}\label{c:quasi additive monotonicity}
For every $\mathfrak{m}_0>0$ there exist constants $\varrho, C>0$ depending on
$\mathfrak{m}_0$ and $[\mA]_{0,\alpha}$ with this property. If $u$ is a solution to \eqref{e:prob} under assumptions (H1)-(H3), for every $x_0 \in \Gamma(u)\cap B_{\sfrac12}'$ satisfying (H4) with $\mathfrak{m}(x_0)\leq \mathfrak{m}_0$,
if $\max\{(\mathfrak{a}(x_0))^{\sfrac1\alpha},r\}<\varrho$ and
$x\in B'_{\sfrac r2}(x_0)$, then
\begin{equation}\label{e:almost monotonicity additive}
[\max\{(\mathfrak{a}(x_0))^{\sfrac1\alpha},r\},\varrho]\ni
 t\mapsto I_u(x,t)+C t^{\alpha}
\quad\text{is non-decreasing}.
\end{equation}
In particular, if $\mathfrak{a}(x_0)=0$, then
\begin{equation}\label{e:freq lb bis}
I_u(x_0,0^+):=\lim_{r\downarrow 0}I_u(x_0,r)\geq\sfrac32\,.
\end{equation}
\end{corollary}

\begin{proof}
The proof of \eqref{e:almost monotonicity additive} is straighforward from inequality \eqref{e:Iuprime}, estimate \eqref{e:stimaRut} in Proposition~\ref{p:monotonia freq} and \eqref{e:I limitato} in Lemma \ref{l:lim uniforme}.
Furthermore, if $\mathfrak{a}(x_0)=0$, then from \cite[Lemma 1]{ACS08}, Corollary~\ref{c:compactness} and
 Proposition~\ref{p:monotonia freq} one deduces that
 $I_u(x_0,r)\geq\frac32e^{-C r^\alpha}$ for all $r\in(0,\varrho]$.
 Therefore, $I_u(x_0,0^+)
 \geq\sfrac32$.
 \end{proof}
{\begin{remark}
The $H^2_{\loc}(B_1)$ regularity of a solution $u$ has been exploited to infer the quasi-monotonicity property of the frequency function $I_u(x_0,\cdot)$ at points $x_0$ as in the statement of Corollary \ref{c:quasi additive monotonicity} in order to estimate the error term $\epsilon_{D'}$ (cf. Lemma \ref{l:stima J1}). 
Different approaches, such as that in \cite{JP,JPSm}, lead to quasi-monotonicity formulas holding in the less restrictive H\"older regularity scale for the matrix field. Despite this, in
order to establish rectifiability of free boundary points with finite order of contact we shall crucially use the $W^{1,p}$ regularity of the matrix field as well as the already mentioned $H^2_{\loc}(B_1)$ regularity of solutions (cf. Proposition \ref{p:D_x frequency}).
\end{remark}}

%
%

\section{Main estimates on the frequency}\label{s:frequency estimate}

\subsection{Oscillation estimate of the frequency}
We introduce the following notation for the radial variations of the frequency
at a point $x_0 \in \Gamma(u)$ with $\mathfrak{m}(x_0)<\infty$ and $\mathfrak{a}(x_0)=0$:
\begin{equation}\label{e:spatial osc freq}
\Delta^r_{\rho}(x_0) := I_u(x_0,r)+Cr^\alpha - \big(I_u(x_0,\rho)+C\rho^\alpha),\qquad 0<\rho<r,
\end{equation}
for $C>0$ the constant in Corollary~\ref{c:quasi additive monotonicity}, so that
$\Delta^r_{\rho}(x_0)\geq 0$ for $r,\rho$ sufficiently small.
The result in the ensuing Proposition~\ref{p:D_x frequency} shows how the spatial
oscillation of the frequency in two nearby points at a given scale is in turn
controlled by the radial variations at comparable scales.
We establish first a technical result. To this aim it is convenient to define the parameter
\begin{equation}\label{e:theta}
\theta:=\min\{[\mA]_{0,\alpha}^{-\sfrac1\alpha},1\}.
\end{equation}

\begin{lemma}\label{l:monotonia}
{For every $\mathfrak{m}_0>0$ there exist constants $\varrho, C>0$ depending on
$\mathfrak{m}_0$ and $[\mA]_{0,\alpha}$ with this property. If $u$ is a solution to \eqref{e:prob} under assumptions (H1)-(H3), for every $x_0 \in \Gamma(u)\cap B_{\sfrac12}'$ satisfying (H4) with $\mathfrak{m}(x_0)\leq \mathfrak{m}_0$,} 
and $\mathfrak{a}(x_0)=0$, then for all $r_1\leq \varrho$, $r_0\in (\frac{\theta}{16}r_1, r_1)$, 
and $x\in {B_{\frac{\theta}{32} r_1}'(x_0)}$, we have
\begin{align}\label{e:monotonia con resto}
\int_{B_{r_1}(x)\setminus B_{r_0}(x)} 
\big(\nabla u(z) \cdot (z - x) - & I_u(x,r_0)\,u(z)\big)^2 {\textstyle{\frac1{|z-x|}}}\,\d z
\leq C H_u(x,2\,r_1)\,\Delta^{2r_1}_{r_0}(x). 
\end{align}
\end{lemma}

\begin{proof}
By translation, it suffices to prove the lemma for $x_0 =\underline{0}$. We start off with the following computation that uses Remark \ref{r:monotonia freq}:
\begin{align}\label{e:I mon}
-\int\phi'&\big(\textstyle{\frac{|z-x|}{t}}\big)
\big(\nabla u(z) \cdot (z-x) - I_u(x,t)\,u(z)\big)^2 
\frac{1}{|z-x|} \d z
\notag\allowdisplaybreaks\\
&= t^2 E_u(x,t) - 2\, t\,I_u(x,t)\,G_u(x,t) + I_u^2(x,t) H_u(x,t)\notag
\allowdisplaybreaks\\
&= \frac{t^2}{H_u(x,t)}\,\big(E_u(x,t) H_u(x,t) - D_u^2(x,t)\big)-2t\epsilon_D(x, t)\, I_u(x,t)\notag\\
&\stackrel{\eqref{e:Iuprime bis}}{=}\frac{t}{2}\,H_u(x,t)\big(I_u'(x,t)-\overline{R}_u(x,t)\big),
\end{align}
where 
\[
\overline{R}_u = \widetilde{R}_u +\frac{4}{H_u(x,t)}\, \epsilon_D(x,t) \, I_u(x,t)
= \frac{\epsilon_{D'}(x,t)}{H_u(x,t)}+ 2I_u(x,t)\frac{\epsilon_D(x,t)}{H_u(x,t)}\,.
\]
with $\widetilde{R}_u$ function in \eqref{e.Rtilde}. 
In particular, the above equalities show that the last factor in \eqref{e:I mon} is nonnegative, {being nonnegative the term on the first line of \eqref{e:I mon} itself}. Therefore,
we may use the elementary integral estimate 
\begin{align}\label{e:fubini}
\int_{B_{r_1}(x)\setminus B_{r_0}(x)} f(z) \d z& \leq -
\frac{C}{r_0}\int_{r_0}^{2r_1} \int \phi'\big(\textstyle{\frac{|z-x|}{t}}\big)
f(z)\, \d z\,\d t
\quad\text{ for all $0<r_0\leq r_1$},
\end{align}
that holds true for any measurable function $f\geq0$, in order to deduce
\begin{align}\label{e:prima stima0}
&\int_{B_{r_1}(x)\setminus B_{r_0}(x)} 
\big(\nabla u(z) \cdot (z-x) - I_u(x,r_0)\,u(z)\big)^2 {\textstyle{\frac{1}{|z-x|}}}\,\d z\notag
\allowdisplaybreaks\\ & 
\stackrel{\eqref{e:fubini}}{\leq}-\frac{C}{r_0}\int_{r_0}^{2r_1} \int
\phi'\big(\textstyle{\frac{|z-x|}{t}}\big)
\big(\nabla u(z) \cdot (z-x) - I_u(x,r_0)\,u(z)\big)^2 {\textstyle{\frac{1}{|z-x|}}}\,\d z\,\d t\notag
\allowdisplaybreaks\\&
\leq -\frac{C}{r_0}\int_{r_0}^{2r_1} \int
\phi'\big(\textstyle{\frac{|z-x|}{t}}\big)
\big(\nabla u(z) \cdot (z-x) - I_u(x,t)\,u(z)\big)^2\,{\textstyle{\frac{1}{|z-x|}}}\d z\d t\notag\\
&\quad -\frac{C}{r_0}\int_{r_0}^{2r_1} \int
\phi'\big(\textstyle{\frac{|z-x|}{t}}\big)
\big(I_u(x,t) - I_u(x,r_0)\big)^2u^2(z)
{\textstyle{\frac{1}{|z-x|}}}\d z\d t\notag
\allowdisplaybreaks\\&
\stackrel{\eqref{e:I mon}}{\leq}\,\frac{C}{r_0}
\int_{r_0}^{2r_1} \frac{t}{2}\, H_u(x,t)\big(I_u'(x,t)-\overline{R}_u(x,t)\big)\, \d t\notag\\
&\qquad+ \frac{C}{r_0}\,
\big((I_u(x,2r_1) - I_u(x,r_0))^2+(2r_1)^{2\alpha}\big)\,\int_{r_0}^{2r_1} H_u(x,t)\,\d t,
\end{align}
where in the last estimate we have used Corollary \ref{c:quasi additive monotonicity} because $x\in {B_{\frac{\theta}{32}r_1}'\subset B_{\sfrac{r_0}{2}}'}$.

From Lemma \ref{l:lim uniforme} and Lemma~\ref{l:monotonia H} we get that $H_u(x,t) \leq C H_u(x,2r_1)$ for all
$t\in[r_0,2r_1]$.
Furthermore, we can use \eqref{e:stimaRut bis} in Remark~\ref{r:monotonia freq} to estimate $|\overline{R}_u(x,t)|$ for all $x\in B_{\frac{\theta}{32}r_1}(x_0)$ and $t\in[r_0,2r_1]$.
Thus, by taking into account that $I_u'(x,t)-\overline{R}_u(x,t)\geq 0$ (cf. \eqref{e:I mon}), 
from \eqref{e:prima stima0} we get 
\begin{align*}
\int_{B_{r_1}(x)\setminus B_{r_0}(x)} & \big(\nabla u(z) \cdot (z-x) - I_u(x,r_0)\,u(z)\big)^2 
{\textstyle{\frac{1}{|z-x|}}}\,\d z\notag\\ 
& \leq C
H_u(x,2r_1) \int_{r_0}^{2r_1}\big(I_u'(x,t)-\overline{R}_u(x,t)\big)\, \d t\notag\\
&+C
\,H_u(x,2r_1)\Big(\big(I_u(x,2r_1) - I_u(x,r_0)\big)^2+(2r_1)^\alpha-r_0^\alpha\Big)\notag\\
& \leq C\,H_u(x,2r_1)\,\big(I_u(x,2r_1) - I_u(x,r_0)\big)
+CH_u(x,2r_1)\big((2r_1)^\alpha-r_0^\alpha\big)\\
&\leq C\,H_u(x,2r_1)\,\Delta_{r_0}^{2r_1}(x),
\end{align*}
where we used that $r_0\in (\frac{\theta}{16}r_1,r_1)$ and $I_u(x,t) \leq C$ for all $t\in (r_0, 2r_1)$ and $x\in B'_{\frac{\theta}{32}r_1}$ by Lemma~\ref{l:lim uniforme}.
\end{proof}

We are now ready to prove a spatial oscillation estimate on the frequency function in terms of the radial oscillation computed between suitable radii in all points belonging to a neighbourhood of a point $x_0$ 
with $\mathfrak{a}(x_0)=0$.

\begin{proposition}\label{p:D_x frequency}
{For every $\mathfrak{m}_0>0$ there exist constants $\varrho, C>0$ depending on
$\mathfrak{m}_0$ and $[\mA]_{0,\alpha}$ with this property. If $u$ is a solution to \eqref{e:prob} under assumptions (H1)-(H3), then for every $x_0 \in \Gamma(u)\cap B_{\sfrac12}'$ satisfying (H4) with $\mathfrak{m}(x_0)\leq \mathfrak{m}_0$,} 
and $\mathfrak{a}(x_0)=0$, for all $R, \rho>0$ with $R>\frac{32}{\theta}$ and $R\rho <\varrho$, we have 
\begin{equation}\label{e:D_x frequency}
\big\vert I_u(x_1, R\rho) - I_u(x_2, R\rho)\big\vert
\leq C \,
\Big(\big(\Delta^{4R\rho}_{\sfrac{R \rho}{4}}(x_1)\big)^{\sfrac{1}{2}}
+\big(\Delta^{R\rho}_{\sfrac{R\rho}{4}}(x_2)\big)^{\sfrac{1}{2}}+(R\rho)^\alpha\Big),
\end{equation}
for every $x_1, x_2 \in B'_{\rho}(x_0)$ {(where $\theta$ is defined in \eqref{e:theta})}.
\end{proposition}
\begin{proof}
 Without loss of generality, we show the conclusion for $x_0=\underline{0}$. Moreover, we set $t=R\rho$.
 Note that, under the assumption $R>\frac{32}{\theta}$, if $t <\varrho$ is sufficiently small, then $I_u(x,t)\leq C$ with the constant $C$ depending on $\mathfrak{m}_0$ by Lemma \ref{l:lim uniforme}.

The proof is based on estimating the tangential derivative
of the frequency function $x\mapsto I_u(x,t)$ 
for $x\in B'_{\rho}$ by taking advantage of the $H^2_\loc$ regularity of $u$.
Thus, we start off noticing that the functions $x\mapsto H_u(x,t)$
and $x\mapsto D_u(x,t)$ are differentiable and, for every 
$e \in \R^{n+1}$ with $e \cdot e_{n+1} = 0$,
we have that
\begin{align}\label{e:dH}
\de_e H_u(x,t) = - 2 \, \int \phi'\big(\textstyle{\frac{|z-x|}{t}}\big)\,
u(z)\, \de_e u(z)\,\frac{1}{|z-x|} \,\d z,
\end{align}
and setting $\mB(z):=\mA(z)-\mA(\underline{0})(=\mA(z)-\textup{Id})$
\begin{align}\label{e:dD}
 \de_e D_u(x,t) & = 2\int 
\phi\big(\textstyle{\frac{|z-x|}{t}}\big)
\nabla u(z)  \cdot \nabla (\de_e u)(z)\,\d z\notag\\
& = 2\int \phi\big(\textstyle{\frac{|z-x|}{t}}\big)
\big(\mA-\mB\big)(z)\nabla u(z)  \cdot \nabla (\de_e u)(z)\,\d z\notag\\
& = - 2\,t^{-1}\int \phi'\big(\textstyle{\frac{|z-x|}{t}}\big)\,
\de_e u(z)\, \mA(z)\nabla u(z) \cdot\frac{z-x}{|z-x|}\,\d z\notag\\
&-2\int \phi\big(\textstyle{\frac{|z-x|}{t}}\big)
\mB(z)\nabla u(z)  \cdot \nabla (\de_e u)(z)\,\d z\notag\\
& = - 2\,t^{-1}\int \phi'\big(\textstyle{\frac{|z-x|}{t}}\big)\,
\de_e u(z)\,\nabla u(z) \cdot\frac{z-x}{|z-x|}\,\d z
+\epsilon_{\partial_eD}(x,t)\,,
\end{align}
where
\begin{align}\label{e:eD}
\epsilon_{\partial_eD}(x,t)&:=- 2\,t^{-1}\int \phi'\big(\textstyle{\frac{|z-x|}{t}}\big)\,
\mB(z)\de_e u(z)\,\nabla u(z) \cdot\frac{z-x}{|z-x|}\,\d z\notag\\
&-2\int \phi\big(\textstyle{\frac{|z-x|}{t}}\big)
\mB(z)\nabla u(z)\cdot \nabla (\de_e u)(z)\,\d z. 
\end{align}
The third equality in \eqref{e:dD} follows from the divergence theorem 
applied to the vector field $V(z):= \phi\big(\textstyle{\frac{|z-x|}{t}}\big)\,\de_e u(z) \mA(z)\nabla u(z)$,
note that $V \in C^{\infty}(B_t(x)\setminus B_1',\R^{n+1})$,
$V$ has compact support and the divergence of $V$ does not concentrate on $B_1'$.
Recalling the $H^2_{\loc}$-estimates in Theorem \ref{t:reg} and the doubling estimates in Corollary \ref{c:freq ben def}, we have
\begin{align}\label{e:stimaJB}
 |\epsilon_{\partial_eD}(x,t)|&\leq C|e|t^{-1}(|x|+t)^\alpha D_u(x,2t)
 +C(|x|+t)^\alpha D_u^{\sfrac12}(x,t)\|u\|_{H^2(B_t(x))}\notag\\
&\leq C|e|t^{-1}(|x|+t)^\alpha\big(D_u(x,2t)+
D_u^{\sfrac12}(x,t)D_u^{\sfrac12}(x,2t)\big)\notag\\&
 \leq C|e|t^{-1}(|x|+t)^\alpha D_u(x,t)\,.
\end{align}
We choose $e := x_2 - x_1$ and set
\[
\mathcal{E}_i(z) := \nabla u(z) \cdot(z-x_i) -  I_u(x_i,t)\,u(z)
\quad\text{for $i=1,2$,}
\]
\[
\Delta I:= I_u(x_1,t) - I_u(x_2,t)
\and
\Delta\mathcal{E}(z):=\mathcal{E}_1(z) - \mathcal{E}_2(z).
\]
Then, we have that $\de_e u(z) = \Delta I \; u(z) + \Delta\mathcal{E}(z)$.
Thus, from \eqref{e:dH} we infer that 
\begin{align*}
\de_e H_u(x,t) &= 2\Delta I \cdot H_u(x,t)
-2 \int \phi'\big(\textstyle{\frac{|z-x|}{t}}\big)  \, 
\Delta\mathcal{E}(z) \;\frac{u(z)}{|z-x|}\,\d z\,,
\end{align*}
while from \eqref{e:dD} and \eqref{e:variationG} we conclude
\begin{align*}
\de_e D_u(x,t) &= 2\,\Delta I\cdot (D_u(x,t)+\epsilon_D(x,t))+\epsilon_{\partial_eD}(x,t)\\
&-2\,t^{-1} \int \phi'\big(\textstyle{\frac{|z-x|}{t}}\big)  \, 
\Delta\mathcal{E}(z) \;\nabla u (z) \cdot \frac{z-x}{|z-x|}\,\d z\,.
\end{align*}
In particular, by a direct computation we deduce that
\begin{align}\label{e:derivata2}
\de_e I(x,t) & =
\frac{t}{H_u^2(x,t)}\,\big(
H_u(x,t) \, \de_e D_u(x,t) - D_u(x,t)\, \de_e H_u(x,t) \big)\notag\\ 
& = \frac{2}{H_u(x,t)}\int -\phi'\big(\textstyle{\frac{|z-x|}{t}}\big)  \, 
\Delta\mathcal{E}(z)\;\Big(\nabla u (z) \cdot (z-x)-
I_u(x,t)\, u(z)\Big)\,\frac{1}{|z-x|}\,\d z\notag\\
&+\frac{t}{H_u(x,t)}\cdot (2 \Delta I\cdot \epsilon_D(x,t)+\epsilon_{\partial_eD}(x,t)).
\end{align}

We estimate \eqref{e:derivata2} (recall that $t = R\rho$ and $x \in B_\rho'$).
First notice that thanks to \eqref{e:variationD errore} we may conclude that 
\begin{align}\label{e:errore1}
 \frac{t}{H_u(x,t)}|2 \Delta I\cdot \epsilon_{D}(x,t)|\leq C|\Delta I|t^\alpha(I_u(x,t)+I_u^{\sfrac12}(x,t))
 \leq Ct^\alpha\,,
\end{align}
for some $C>0$.
Furthermore, by \eqref{e:stimaJB} we get that
\begin{align}\label{e:errore2}
\frac{t}{H_u(x,t)}|\epsilon_{\partial_eD}(x,t)|\leq Ct^\alpha I_u(x,t)\leq Ct^\alpha\,,
\end{align}
{where we used that $|e|\leq 2\rho\leq t$.}
Note that, since $x\in B_\rho'$,
by elliptic regularity of $u$ (cf. Theorem~\ref{t:reg}) we infer that 
\begin{align*}
\sup_{z\in B_t^+(x)}&|\nabla u (z) \cdot (z-x) - I_u(x,t)\, u(z)|\\
& \leq 
t\sup_{z\in B_{t+\rho}^+} |\nabla u (z)|+ I_u(x,t)\|u\|_{C^0(B_{t+\rho})}\\
&\leq C\,t^{-\frac{n+1}{2}}\,\|u\|_{L^2(B_{2t+2\rho})}
\leq C\,t^{-\frac{n}{2}}\,H_u^{\sfrac12}\big(2t +2\rho\big), 
\end{align*}
where we use \eqref{e:L2 vs H} in Lemma \ref{l:monotonia H}. Hence, we have that
\begin{align}\label{e:derivata2.5}
\de_e I_u(x,t) &\leq C\,t^{-\frac{n}{2}}\frac{H_u^{\sfrac12}\big(2t +2\rho\big)}{H_u(x,t)}
\int -\phi'\big(\textstyle{\frac{|z-x|}{t}}\big)
\big(|\mathcal{E}_1(z)| + |\mathcal{E}_2(z)|\big)
\frac{1}{|z-x|}\,\d z+C\,t^\alpha\,.
\end{align}
In order to estimate the integral term in \eqref{e:derivata2.5}, we notice that 
\[
B_t(x)\setminus B_{\sfrac{t}{2}}(x)
\subset B_{t+2\rho}(x_i) \setminus B_{\sfrac{t}{2}-2\rho}(x_i)
\quad \forall \,x\in B_{\sfrac12}', \;\text{for }\;i=1,2;
\]
therefore
\begin{align}\label{e:conto integrale}
\int_{B_{t}(x)\setminus B_{\frac{t}{2}}(x)}
|\mathcal{E}_i(z)|\,\textstyle{\frac{1}{|z-x|}}\,\d z
&\leq{\textstyle{\frac{2\,(t+2\rho)}{t}}}\int_{B_{t+2\rho}(x_i) \setminus B_{\frac{t}{2}-2\rho}(x_i)}
|\mathcal{E}_i(z)| \,\textstyle{\frac{1}{|z-x_i|}}\,\d 
z\notag\\
&\leq C\, t^{\frac{n}{2}}\Big(\int_{B_{t+2\rho}(x_i) \setminus B_{\frac{t}{2}-2\rho}(x_i)}
\mathcal{E}_i^2(z) \,\textstyle{\frac{1}{|z-x_i|}}\,\d z
\Big)^{\sfrac12},
\end{align}
where we choose $R>8$ and we use the direct computation
\[
\int_{B_{t+2\rho}(x_i) \setminus B_{\frac{t}{2}-2\rho}(x_i)}
\textstyle{\frac{1}{|z-x_i|}}\,\d z\leq C\,t^n,
\]
with  $C>0$ a dimensional constant.
If $R>\frac{32}{\theta}$, then we are in the position to apply Lemma~\ref{l:monotonia}
(with $r_0=\sfrac t2-2\rho$ and $r_1=t+2\rho$) to get
\begin{align}
\int_{B_{t+2\rho}(x_i) \setminus B_{\sfrac{t}{2}-2\rho}(x_i)}
\mathcal{E}_i^2(z) \,\textstyle{\frac{1}{|z-x_i|}}\,\d z 
\leq C_{\ref{l:monotonia}}(A)\,H_u\big(x_i,2t+4\rho\big)\,
\Delta^{2(t+2\rho)}_{\frac{t}{2}-2\rho}(x_i). 
\label{e:stima secondo fattore}
\end{align}
Using \eqref{e:derivata2.5}-\eqref{e:stima secondo fattore} we claim that for all $x\in B_\rho'$
\begin{align}\label{e:finale con brio}
\de_e I_u(x,t)
&\leq C\,\big(\Delta^{4t}_{\sfrac{t}{4}}(x_1)\big)^{\sfrac12}
+ C\,\big(\Delta^{4t}_{\sfrac{t}{4}}(x_2)\big)^{\sfrac12}+C\,t^\alpha\,,
\end{align}
from which the conclusion follows by integrating \eqref{e:finale con brio} along the segment $\{x_1+ r\,e:\, r\in[0,1]\}$. Indeed, $4t \geq 2(t+2\rho)$ and $\frac t4<\frac{t}{2}-2\rho$, and the monotonicity of Corollary \ref{c:quasi additive monotonicity} in the set of radii under consideration. Moreover,
\[
H_u^{\sfrac12}\big(2t+2\rho\big)\,H_u^{-1}\big(x,t\big)
H_u^{\sfrac12}\big(x_i,2t+4\rho\big)\leq
H_u^{\sfrac12}\big(2t+2\rho\big)\,H_u^{-1}\big(t\big)
H_u^{\sfrac12}\big(2t+4\rho\big)\leq C\,,
\]
thanks to the estimates in Lemma \ref{l:lim uniforme} and Corollary \ref{c:freq ben def}
because $x_i\in B'_{\rho}$ and $t = R\rho \in (2\rho, \varrho)$.
\end{proof}

%
%
\subsection{Estimate of the mean-flatness via the frequency function}
We introduce the mean-flatness.

\begin{definition}\label{d:mean-flat}
Given a Radon measure $\mu$ in $\R^{n+1}$, 
for every $x_0 \in \R^n$ and for every $r>0$, set
\begin{equation}\label{e:beta}
\beta_\mu(x,r) := \inf_{\cL} \Big(
r^{-n-1} \int_{B_r(x)} \dist^2(y,\cL)\d\mu(y)\Big)^{\sfrac{1}{2}},
\end{equation}
where the infimum is taken among all affine $(n-1)$-dimensional planes $\cL \subset \R^{n+1}$, and 
$\dist(y,\cL) := \inf_{x \in \cL} |y-x|$.
\end{definition}

As shown in \cite{FS18-1, FS22} the mean flatness $\beta_\mu$ of an arbitrary measure $\mu$ supported on $\Gamma(u)$ is controlled in terms of the integration of suitable radial oscillations of the frequency with respect to $\mu$.

\begin{proposition}\label{p:mean-flatness vs freq}
{For every $\mathfrak{m}_0>0$ and $R>\frac{64}{\theta}$ (where $\theta$ is defined in \eqref{e:theta}),
there exist constants $\varrho, C>0$ depending on $R$,
$\mathfrak{m}_0$ and $[\mA]_{0,\alpha}$ with this property. If $u$ is a solution to \eqref{e:prob} under assumptions (H1)-(H3), for every $x_0 \in \Gamma(u)\cap B_{\sfrac12}'$ satisfying (H4) with $\mathfrak{m}(x_0)\leq \mathfrak{m}_0$,} 
and $\mathfrak{a}(x_0)=0$, then for every $r>0$ with $Rr<\varrho$, 
for every finite Borel measure $\mu$ with $\spt\mu\subseteq \Gamma(u)$,
and for every $p \in \Gamma(u) \cap B'_{r}(x_0)$
\begin{equation}\label{e:mean-flatness vs freq}
\beta_{\mu}^2 (p,r) \leq 
\frac{C}{r^{n-1}}\Big(
\int_{B_{r}(p)}\Delta_{(R-5)\,\sfrac{r}{2}}^{(2R+4)\,r}(x)\,\d\mu(x)
+(Rr)^\alpha\mu(B_{r}(p))\Big)\,.
\end{equation}
\end{proposition}

\begin{proof}
The proof is a simple adaptation of the ones in \cite[Proposition 4.2]{FS18-1} and \cite[Proposition 5.1]{FS22}. The condition $R>\frac{64}{\theta}$ is used in order to apply Lemma~\ref{l:monotonia}. We leave the details to the readers.

\end{proof}

%
%

\section{Intrinsic frequency}\label{s.intrinsic frequency}

In this section we introduce an elementary change of variables in order to make a generic free boundary point $x_0$ satisfy $\frak{a}(x_0)=0$ for a different, related thin obstacle problem {(cf. \cite{GS14,GPS16,GPS18} for the thin obstacle, and \cite{FoGeSp15} in the case of the classical obstacle problem)}.
In such a way we define an intrinsic frequency function for which the conclusions of Proposition~\ref{p:mean-flatness vs freq} hold even without the {matrix $\mathbb{A}(x_0)$} being the identity at free boundary points.

Given a solution $u$ of \eqref{e:ob-pb local}, and $x_0\in\Gamma(u)\cap B_1'$, consider the function 
$u_{\mathbb{A}(x_0)}:\PPhi^{-1}(B_1)\to\mathbb{R}$ defined by
\[
u_{\mathbb{A}(x_0)}(x):=u(\PPhi(x))\,,
\]
where $\PPhi:\mathbb{R}^{n+1}\to\mathbb{R}^{n+1}$ is the 
affine map $\PPhi(x):=x_0+\mathbb{A}^{\sfrac12}(x_0)(x-x_0)$. 
In particular, changing variables by means of $\PPhi$ leads to
\begin{equation}\label{e:cambio di coordinate3}
\int_{B_1}\langle\mathbb A(x)\nabla u(x) , \nabla u(x) \rangle\,\d x=\det\big(\LL(x_0)\big)\,\mathscr{E}_{\mathbb{A}(x_0)}\big(u_{\mathbb{A}(x_0)}, 
\PPhi^{-1}(B_1)\big)\,,
\end{equation}
where 
\begin{equation}\label{e:enrgA}
\mathscr{E}_{\mathbb{A}(x_0)}(v,U):=\int_{U}
\langle \mathbb{C}_{x_0}(x)\nabla v(x),\nabla v(x)\rangle\d x\,,
\end{equation}
for every open set $U\subseteq
\PPhi^{-1}(B_1)$,
and $\mathbb{C}_{x_0}(x):=\mathbb{A}^{-\sfrac12}(x_0)\mathbb{A}(
\PPhi(x)) \mathbb{A}^{-\sfrac12}(x_0)$.
Note that $\mathbb{C}_{x_0}(x_0)=\text{Id}$.
Therefore, $u_{\mathbb{A}(x_0)}$ turns out to be the solution of 
the thin obstacle problem for the energy in \eqref{e:enrgA}
among all functions in $v\in g(\PPhi)
+H^1_0\big(
\PPhi^{-1}(B_1)\big)$ that are even across the corresponding hyperplane
$\PPhi^{-1}(\{x_{n+1}=0\})=\{x_{n+1}=0\}$ (thanks to hypothesis (H3)), and such that
$v|_{\PPhi^{-1}(B_1')}\geq 0$.
Moreover, there is a bijection of the free boundaries:
$\Gamma(u_{\mathbb{A}(x_0)})=
\PPhi^{-1}(\Gamma(u))$.

Let $u$ be a solution to \eqref{e:ob-pb local}, and let 
$x_0\in B_1'$ and $r>0$ be such that $\Phi_{x_0}(B_r(x_0)\cap \{x_{n+1}=0\})\subset B_1'$.
Being $u_{\mathbb{A}(x_0)}$ solution to the thin obstacle problem corresponding to the matrix field $\mathbb{C}_{x_0}(\cdot)$ in $B_r(x_0)$,
we consider the related frequency function
\begin{equation}\label{e:intrinsic I}
I_{u_{\mathbb{A}(x_0)}}(x,s)=\frac{s\, D_{u_{\mathbb{A}(x_0)}}(x,s)}{H_{u_{\mathbb{A}(x_0)}}(x,s)}\,,
\qquad x\in \{x_{n+1}=0\}\cap B_{r}(x_0), \quad s< r - |x-x_0|.
\end{equation}
In passing, note that if $x_0$ satisfies assumption $\frak{a}(x_0)=0$, then $u$ coincides with
$u_{\mathbb{A}(x_0)}$, and correspondingly $I_{u_{\mathbb{A}(x_0)}}$ coincides with $I_u$ at all points in $B_1'$ and admissible radii. 
For later purposes it is convenient to point out explicit formulas for the Dirichlet energy \begin{equation}\label{e:intrinsic D}
D_{u_{\mathbb{A}(x_0)}}(x,s) =
 \int \phi\Big(\textstyle{\frac{|\Phi^{-1}_{x_0}(y)-x|}{s}}\Big)
 \frac{\langle\mathbb{A}(x_0)\nabla u(y),\nabla u(y)\rangle}{\det\LL(x_0)}\,\d y,
\end{equation}
and for the ``boundary'' $L^2$ norm of $u$
\begin{equation}\label{e:intrinsic H}
H_{u_{\mathbb{A}(x_0)}}(x,s)=
 -\int \phi'\Big(\textstyle{\frac{|\Phi^{-1}_{x_0}(y)-x|}{s}}\Big)
 \,\frac{u^2(y)}{|\Phi^{-1}_{x_0}(y)-x|}\frac{1}{\det\LL(x_0)}\,\d y.
\end{equation}
We call $I_{u_{\mathbb{A}(x_0)}}(x_0,r)$ the intrinsic function at a free boundary point $x_0\in \Gamma(u)$ and set
\[
N_u(x_0,r) := I_{u_{\mathbb{A}(x_0)}}(x_0,r).
\]
{Having fixed a point $x_0$ of the free boundary with finite frequency,} we compare the intrinsic frequency function and the (standard) Dirichlet based one {at points of the free boundary  with finite frequency close to $x_0$. Thus, for every $\mathfrak{m}_0>0$ we set
\[
\Gamma^{\mathfrak m_0}(u):=\{x\in \Gamma(u)\cap B_{\sfrac12}':\,\sup_{r\in(0,\sfrac12)}N_u(x,r)\leq \mathfrak{m}_0\}\,.
\]

\begin{proposition}\label{p:comparison frequencies}
For every $\mathfrak{m}_0>0$ there exist constants $\varrho, C>0$ depending on the ellipticity constant $\lambda$, $\mathfrak{m}_0$ and $[\mA]_{0,\alpha}$ with this property. If $u$ is a solution to \eqref{e:prob} under assumptions (H1)-(H3), for every $x_0,\,x_1 \in \Gamma^{\mathfrak m_0}(u)\cap B_{\sfrac12}'$,
if $r\in(0,\varrho)$ and $|x_0-x_1|<C^{-1}r$, then
\begin{align}\label{e:comparison I}
\big|I_{u_{\mathbb{A}(x_0)}}(\PPhi^{-1}(x_1),r)-
N_u(x_1,r)\big|\leq
C|x_0-x_1|^{\sfrac\alpha2}
N_u(x_1,r)\,.
\end{align}
\end{proposition}
\begin{proof}
We start off noticing that by (H1)
\[
|\PPhi^{-1}(x_1)-x_0|=|\PPhi^{-1}(x_1)-\PPhi^{-1}(x_0)|\leq \lambda^{-\sfrac12}|x_1-x_0|\,,
\]
where the square root of the ellipticity constant $\lambda$ estimates from below the norm of $\mathbb{A}^{-\sfrac12}(x_0)$.
Therefore, $I_{u_{\mathbb{A}(x_0)}}(\PPhi^{-1}(x_1),r)$ is well defined provided that 
$\varrho, C^{-1}$ are small (cf. Lemma~\ref{l:lim uniforme}).

To prove the inequality in \eqref{e:comparison I} it is convenient to recall formulas \eqref{e:intrinsic I}-\eqref{e:intrinsic H}:
\[
D_{u_{\mathbb{A}(x_0)}}(\PPhi^{-1}(x_1),r)=
\int {\textstyle{\phi\Big(\frac{|\mathbb{A}^{-\sfrac12}(x_0)(y-x_1)|}{r}}\Big)\frac{\langle\mathbb{A}(x_0)\nabla u(y),\nabla u(y)\rangle}{\det\LL(x_0)}}\,\d y\,,
\]
and
\[
H_{u_{\mathbb{A}(x_0)}}(\PPhi^{-1}(x_1),r)=-\int {\textstyle{
\phi'\Big(\frac{|\mathbb{A}^{-\sfrac12}(x_0)(y-x_1)|}{r}\Big)
\,\frac{u^2(y)}{|\mathbb{A}^{-\sfrac12}(x_0)(y-x_1)|}\frac{1}{\det\mathbb{A}^{\sfrac12}(x_0)}}}\,\d y\,.
\]
To estimate the difference between the Dirichlet energies we introduce the sets
\begin{align*}
U_r(x)&:=\Big((x+\mathbb{A}^{\sfrac12}(x)B_r)\cup
({x_1}+\mathbb{A}^{\sfrac12}(x_0)B_r)\Big)\setminus\\
&\qquad \setminus\Big((x+\mathbb{A}^{\sfrac12}(x)B_{\sfrac r2})\cap
({x_1}+\mathbb{A}^{\sfrac12}(x_0)B_{\sfrac r2})\Big)
\end{align*}
for all $x\in B_{\sfrac 12}'$. Then, we argue as follows
 \begin{align}\label{e:comparison D}
& D_{u_{\mathbb{A}(x_1)}}(x_1,r)-D_{u_{\mathbb{A}(x_0)}}(\PPhi^{-1}(x_1),r)\notag\\&= 
\int {\textstyle{\phi\Big(\frac{|\mathbb{A}^{-\sfrac12}(x_1)(y-x_1)|}{r}}\Big)\frac{\langle\mathbb{A}(x_1)\nabla u(y),\nabla u(y)\rangle}{\det\LL(x_1)}}\,\d y\notag\\
&\quad
-\int {\textstyle{\phi\Big(\frac{|\mathbb{A}^{-\sfrac12}(x_0)(y-x_1)|}{r}}\Big)
\frac{\langle\mathbb{A}(x_0)\nabla u(y),\nabla u(y)\rangle}{\det\LL(x_0)}}\,\d y
\notag\\&= 
\int {\textstyle{\Big(\phi\Big(\frac{|\mathbb{A}^{-\sfrac12}(x_1)(y-x_1)|}{r}}\Big)-\phi\Big(\frac{|\mathbb{A}^{-\sfrac12}(x_0)(y-x_1)|}{r}\Big)\Big)
\frac{\langle\mathbb{A}(x_1)\nabla u(y),\nabla u(y)\rangle}{\det\LL(x_1)}}\,\d y\notag\\&
\quad+ \int {\textstyle{\phi\Big(\frac{|\mathbb{A}^{-\sfrac12}(x_0)(y-x_1)|}{r}\Big)\Big(\frac{\langle\mathbb{A}(x_1)\nabla u(y),\nabla u(y)\rangle}{\det\LL(x_1)}
-\frac{\langle\mathbb{A}(x_0)\nabla u(y),\nabla u(y)\rangle}{\det\LL(x_0)}\Big)}}\,\d y\notag\\
&=:D^{(1)}(r)+D^{(2)}(r).
 \end{align}
Since $|y-x_1|\leq C(\Lambda)r$ 
for all $y\in U_r(x_1)$ and $\Phi_{x_1}^{-1}(U_r(x_1))\subseteq 
B_{C(\lambda,\Lambda)r}(x_1)$, we deduce that
\begin{align}\label{e:comparison D1}
|D^{(1)}(r)|&\leq
C[\phi]_{0,1} |\mathbb{A}^{-\sfrac12}(x_0)-\mathbb{A}^{-\sfrac12}(x_1)|
 \int_{U_r(x_1)}{\textstyle{
 \frac{\langle\mathbb{A}(x_1)\nabla u(y),\nabla u(y)\rangle}{\det\LL(x_1)}}} \,\d y\notag\\
  &\leq   C |\mathbb{A}^{\sfrac12}(x_0)-\mathbb{A}^{\sfrac12}(x_1)|
\int_{\Phi_{x_1}^{-1}(U_r(x_1))}
{\textstyle{|\nabla u_{\mathbb{A}(x_1)}(y)|^2}}\,\d y
\notag\\
&
\stackrel{(H1)}{\leq} C|x_1-x_0|^{\sfrac\alpha2} 
D_{u_{\mathbb{A}(x_1)}}\big(x_1,
Cr\big)
\end{align}
and, analogously,
\begin{align}\label{e:comparison D2}
|D^{(2)}(r)|&\leq
\Big|{\textstyle{\text{Id}-
\frac{\det\LL(x_1)}{\det\LL(x_0)}\mathbb{A}^{-1}(x_1)\mathbb{A}(x_0)
}}\Big|\int_{U_r(x_1)}{\textstyle{
 \frac{\langle\mathbb{A}(x_1)\nabla u(y),\nabla u(y)\rangle}{\det\LL(x_1)}}} \,\d y\notag\\
&\stackrel{(H1)}{\leq}C|x_1-x_0|^{\sfrac\alpha2}
D_{u_{\mathbb{A}(x_1)}}\big(x_1,C
r\big)\,,
\end{align}
for some constant $C=C(n,\lambda,\Lambda,[\phi]_{0,1},[\mathbb{A}]_{0,\alpha})>0$.
For $\rho$ sufficiently small we apply iteratively Corollary~\ref{c:freq ben def} to $u_{\mathbb{A}(x_1)}$ to conclude
\begin{equation}\label{e:stima D}
|D_{u_{\mathbb{A}(x_1)}}(x_1,r)-
D_{u_{\mathbb{A}(x_0)}}(\PPhi^{-1}(x_1),r)|\leq C
|x_1-x_0|^{\sfrac\alpha2}D_{u_{\mathbb{A}(x_1)}}(x_1,r)\,.
\end{equation}

To estimate the difference of the $H$-terms we
define $[0,\infty)\ni t\mapsto\psi(t):=\sfrac{\phi'(t)}t$ (recall that 
$\phi'(t)=0$ for $t\in[0,\sfrac12]\cup[1,\infty)$) and notice that 
$\psi$ is Lipschitz continuous on $[0,\infty)$, having assumed 
$\phi\in C^{1,1}$. We then argue as follows 
 \begin{align}\label{e:comparison H}
&H_{u_{\mathbb{A}(x_1)}}(x_1,r)-H_{u_{\mathbb{A}(x_0)}}(\PPhi^{-1}(x_1),r)\notag\\ 
&=\int {\textstyle{
\psi\Big(\frac{|\mathbb{A}^{-\sfrac12}(x_0)(y-x_1)|}{r}\Big)
\,\frac{u^2(y)}{r\det\mathbb{A}^{\sfrac12}(x_0)}}}\,\d y
-\int {\textstyle{
\psi\Big(\frac{|\mathbb{A}^{-\sfrac12}(x_1)(y-x_1)|}{r}\Big)
\,\frac{u^2(y)}{r\det\mathbb{A}^{\sfrac12}(x_1)}}}\,\d y
\notag\\
&=\int\Big( {\textstyle{
\psi\Big(\frac{|\mathbb{A}^{-\sfrac12}(x_0)(y-x_1)|}{r}\Big)-
\psi\Big(\frac{|\mathbb{A}^{-\sfrac12}(x_1)(y-x_1)|}{r}\Big)
\Big)\,\frac{u^2(y)}{r\det\mathbb{A}^{\sfrac12}(x_0)}}}\,\d y
\notag\\
&+\int \psi\Big(\textstyle{\frac{|\mathbb{A}^{-\sfrac12}(x_1)(y-
x_1)|}{r}}\Big)\,\Big(\frac{1}{\det\mathbb{A}^{\sfrac12}(x_0)}-
\,\frac{1}{\det\mathbb{A}^{\sfrac12}(x_1)}\Big)
\frac{u^2(y)}{r}\,\d y=:H^{(1)}(r)+H^{(2)}(r).
 \end{align}
Therefore, we get straightforwardly that 
 \begin{align}\label{e:comparison H2}
|H^{(2)}(r)|\leq \Big|{\textstyle{
\frac{\det\mathbb{A}^{\sfrac12}(x_1)}{\det\mathbb{A}^{\sfrac12}(x_0)}}}-1\Big| H_{u_{\mathbb{A}(x_1)}}(x_1,r)
\leq C|x_1-x_0|^{\sfrac\alpha2} H_{u_{\mathbb{A}(x_1)}}(x_1,r)\,,
 \end{align}
with $C(n,\lambda,\Lambda,[\mathbb{A}]_{0,\alpha})>0$.
To estimate $H^{(1)}$ we introduce the set
\[
V_r(x_1):=\big(x_1+\mathbb{A}^{\sfrac12}(x_0)(B_r\setminus B_{\sfrac r2})\big)\cup
\big(x_1+\mathbb{A}^{\sfrac12}(x_1)(B_r\setminus B_{\sfrac r2})\big)\,,
\]
and get  
\begin{align*}
 |H^{(1)}(r)|&\leq C
|\mathbb{A}^{-\sfrac12}(x_1)-\mathbb{A}^{-\sfrac12}(x_0)|\int_{V_r(x_1)}{\textstyle{\frac{u^2(y)}{r}}}\,\d y\\
&\leq C|x_1-x_0|^{\sfrac\alpha2}\int_{V_r(x_1)}{\textstyle{\frac{u^2(y)}{r}}}\,\d y
\leq C|x_1-x_0|^{\sfrac\alpha2}\int_{\Phi_{x_1}^{-1}(V_r(x_1))}
{\textstyle{\frac{u^2_{\mathbb{A}(x_1)}(z)}{r}}}\,\d z\,,
\end{align*}
where $C(n,\lambda,\Lambda,[\phi]_{1,1},[\mathbb{A}]_{0,\alpha})>0$. 
We have that 
\[
\Phi_{x_1}^{-1}(V_r(x_1))\subseteq 
B_{(\lambda^{-1}\Lambda)^{\sfrac12}r}(x_1)\setminus
B_{\frac12(\lambda\Lambda^{-1})^{\sfrac12} r}(x_1)
\] 
and thus we may estimate the r.h.s. above as follows
\[
 |H^{(1)}(r)|\leq C|x_1-x_0|^{\sfrac\alpha2}
\int_{B_{(\lambda^{-1}\Lambda)^{\sfrac12}r}(x_1)\setminus
B_{\frac12(\lambda\Lambda^{-1})^{\sfrac12} r}(x_1)}
{\textstyle{\frac{u^2_{\mathbb{A}(x_1)}(z)}{r}}}\,\d z\,.
\]
Being $u_{\mathbb{A}(x_1)}$ a solution to a thin obstacle problem with
$\mathfrak{a}(x_1)=0$, for $\varrho$ sufficiently small Lemma~\ref{l:monotonia H} yields
 \begin{align*}
\int_{B_{(\lambda^{-1}\Lambda)^{\sfrac12}r}(x_1)\setminus B_{\frac12(\lambda\Lambda^{-1})^{\sfrac12} r}(x_1)}
u^2_{\mathbb{A}(x_1)}(z)\,\d z\,
\leq  C r H_{u_{\mathbb{A}(x_1)}}(x_1,
(\lambda^{-1}\Lambda)^{\sfrac12}r\big)\,
 \end{align*}
with $C=C([\mathbb{A}]_{0,\alpha},\mathfrak{m}(x_1))>0$.
In turn, the doubling properties of $H_{u_{\mathbb{A}(x_1)}}(x_1,\cdot)$ together with the quasi-monotonicity in \eqref{e:monotonia H} imply
\begin{align*}
 |H^{(1)}(r)|&\leq C|x_1-x_0|^{\sfrac\alpha2}
 H_{u_{\mathbb{A}(x_1)}}(x_1,r\big)\,.
\end{align*}
 Thus, we conclude that
\begin{equation}\label{e:stima H}
|H_{u_{\mathbb{A}(x_1)}}(x_1,r)-
H_{u_{\mathbb{A}(x_0)}}(\PPhi^{-1}(x_1),r)|\leq C
|x_1-x_0|^{\sfrac\alpha2}H_{u_{\mathbb{A}(x_1)}}(x_1,r)\,,
\end{equation}
and from estimates \eqref{e:stima D} and \eqref{e:stima H}
we conclude (always under the hypothesis that $\varrho$ is sufficiently small)
\[
\big|
I_{u_{\mathbb{A}(x_0)}}(\PPhi^{-1}(x_1),r)-
I_{u_{\mathbb{A}(x_1)}}(x_1,r)\big|\leq 
\frac{2C|x_1-x_0|^{\sfrac\alpha2}}{1-C|x_1-x_0|^{\sfrac\alpha2}}I_{u_{\mathbb{A}(x_1)}}(x_1,r)\,.\qedhere
\]
\end{proof}

In view of the previous estimate on the intrinsic frequency function, we can in rephrase the bound on the mean-flatness in terms of the intrinsic frequency itself dispensing with the assumption 
$\mathfrak{a}(x_0)=0$ for the base point.

For points $x_0 \in \Gamma(u)$ with ${\sup_{r\in(0,\sfrac12)}N_u(x_0,r)<\infty}$ we set
\begin{equation}\label{e:spatial osc freq intrinsic}
\Xi^r_{\rho}(x_0) := N_u(x_0,r)+Cr^\alpha - \big(N_u(x_0,\rho)+C\rho^\alpha),\qquad 0<\rho<r,
\end{equation}
for $C>0$ the constant in \eqref{e:spatial osc freq} and $r$ sufficiently small.
{Finally, we note that the semi-norm $[(\det\mA^{\sfrac 12}(x_0))\mathbb{C}_{x_0}]_{0,\alpha}$ is uniformly bounded; therefore, the constants $\theta$ in \eqref{e:theta} for this new matrices are uniformly bounded with respect to $x_0$ in view of (H1) and (H2). We denote by $\theta_0>0$ its infimum.}

\begin{proposition}\label{p:mean-flatness vs freq bis}
{For every $\mathfrak{m}_0>0$ and $R>\sfrac{64}{\theta_0}$ there exist constants $\varrho, C>0$ depending on $R$, $\mathfrak{m}_0$ and $[\mA]_{0,\alpha}$, with this property.
If $u$ is a solution to \eqref{e:prob} under assumptions (H1)-(H3), for every finite Borel measure $\mu$ 
with $\spt\mu\subseteq\Gamma^{\mathfrak{m}_0}(u)$, for every $x_0 \in
\Gamma^{\mathfrak m_0}(u)$, then for every $r>0$ with $Rr<\varrho$ and $p \in \Gamma^{\mathfrak{m}_0}(u) \cap B'_{r}(x_0)$}
\begin{equation}\label{e:mean-flatness vs freq bis}
\beta_{\mu}^2 (p,r) \leq\frac{C}{r^{n-1}}
\Big(\int_{B_{R_2r}(p)}
\Xi^{R_2 r}_{R_1r}(x)\,\d\mu(x)
+{(R_2r)^{\alpha/2}}\mu(B_{R_2r}(p))\Big)\,,
\end{equation}
{for every $R_2>\max\{2R^2, 2R+4\}$ and $R_1<\frac12(R-5)r$.}
\end{proposition}
\begin{proof}
Set $\mu_{\mathbb{A}(x_0)}:=(\PPhi^{-1})_\#\mu$, then $\spt(\mu_{\mathbb{A}(x_0)})\subseteq
\Gamma(u_{\mathbb{A}(x_0)})$. Note that 
$\PPhi^{-1}(B_r(p))=x_0+\mathbb{A}^{-\sfrac12}(x_0)B_r(p-x_0)\subseteq B_{(1+2\lambda^{-\sfrac12})r}(p)$. 
Thus, from the very definition of the mean flatness $\beta_\mu$ we infer that for all $p\in B_r'(x_0)$
\begin{align*}
\beta_{\mu}^2(p,r)
&=\inf_{\cL} r^{-n-1} \int_{B_r(p)} \dist^2(y,\cL)\d(\PPhi)_\#\mu_{\mathbb{A}(x_0)}(y)\notag\\
&=\inf_{\cL} r^{-n-1} \int_{\PPhi^{-1}(B_r(p))} \dist^2(\PPhi(y),\cL)\d\mu_{\mathbb{A}(x_0)}(y)\notag\\
&=\inf_{\cL} r^{-n-1} \int_{\PPhi^{-1}(B_r(p))}\dist^2(\PPhi(y),\PPhi(\cL))\d\mu_{\mathbb{A}(x_0)}(y)\notag\\
&\leq \Lambda
\inf_{\cL} r^{-n-1} \int_{\PPhi^{-1}(B_r(p))}\dist^2(y,\cL)
\d\mu_{\mathbb{A}(x_0)}(y)\notag\\
&\leq C R^{n+1}\Lambda
\beta_{\mu_{\mathbb{A}(x_0)}}^2(p,R r),
\end{align*}
if {$R\geq 1+2\lambda^{-\sfrac12}$}.
Since, $u_{\mathbb{A}(x_0)}$ satisfies the hypotheses of Proposition~\ref{p:mean-flatness vs freq} 
in $x_0$, recalling that $\mu_{\mathbb{A}(x_0)}=(\PPhi^{-1})_\#\mu$ we deduce that
\begin{align*}
\beta_\mu^2(p,r)&\leq C R^{n+1}\Lambda\,
\beta_{\mu_{\mathbb{A}(x_0)}}^2(p,R r)\\
&\leq\frac{C}{r^{n-1}}
\Big(\int_{B_{Rr}(p)}
\big(\Delta_{u_{\mathbb{A}(x_0)}}\big)^{(2R+4)r}_{\frac12(R-5)r}(x)\d\mu_{\mathbb{A}(x_0)}(x)+(Rr)^\alpha\mu_{\mathbb{A}(x_0)}(B_{Rr}(p))\Big)\\
&=\frac{C}{r^{n-1}} \int_{\PPhi(B_{Rr}(p))}
\big(\Delta_{u_{\mathbb{A}(x_0)}}\big)^{(2R+4)r}_{\frac12(R-5)r}
(\PPhi^{-1}(x))\,\d\mu(x)\\
&+\frac{C}{r^{n-1}}
(Rr)^\alpha\mu(\PPhi(B_{Rr}(p))\,,
\end{align*}
{where we denote by $\Delta_{u_{\mathbb{A}(x_0)}}$ the quantity defined
in \eqref{e:spatial osc freq} by means of $I_{u_{\mathbb{A}(x_0)}}$, {$R\geq (1+2\lambda^{-\sfrac12})\vee\frac{64}{\theta_0}$}, and $r$ is sufficiently small (cf. Proposition~\ref{p:mean-flatness vs freq})}.

Eventually, Proposition~\ref{p:comparison frequencies} provides 
the conclusion as $\spt\mu\subseteq\Gamma^{\mathfrak m_0}(u)$, i.e.,
\begin{align*}
\beta_\mu^2(p,r)
&\leq\frac{C}{r^{n-1}}\Big(\int_{B_{R^2r}(p)}
\Xi^{R_2 r}_{R_1 r}(x)
\,\d\mu(x)+(Rr)^{\alpha/2}\mu(B_{R^2r}(p))\Big)\,,
\end{align*}
for $R_2>\max\{2R^2, 2R+4\}$ and $R_1<\frac12(R-5)r$, 
{and $r$ is sufficiently small (cf. Proposition~\ref{p:comparison frequencies})}.
\end{proof}

\section{The measure and the structure of the free boundary}\label{s:misura}

We recall the definition of homogeneous and almost homegeneous solutions to the (standard) 
thin obstacle problem:
\begin{align*}
\cH := \Big\{ w\in H^1_\loc(\R^{n+1})\setminus\{0\}
 &:\, w(x) = |x|^\lambda\,w\big({\textstyle{\frac{x}{|x|}}}\big),\,\lambda\geq \sfrac{3}{2},\\
&\text{ $w\vert_{B_1}$ solves \eqref{e:prob} with $\mA\equiv \Id$} \Big\},
\end{align*}
Given a solution $u$ to \eqref{e:prob} with $\mA$ satisfying (H1)-(H3), we set
\begin{equation}
\Gamma^{\textup{finite}}(u):= \left\{x\in \Gamma(u) \; :\; \limsup_{r\to 0^+}N_u(x, r)<+\infty\right\}.
\end{equation}
Note that, for every $\mathfrak{m}_0>0$ we have that $\Gamma^{\mathfrak{m}_0}(u)\subseteq \Gamma^{\textup{finite}}(u)$.
For any point $x_0\in \Gamma^{\textup{finite}}(u)$, we set
\[
J_u (x_0,t) := e^{C t^\alpha} N_u(x_0,t),
\]
for all $t>0$ such that $J_u(x_0,t)$ is monotone, namely for all $t\in (0,\varrho)$ with $\varrho>0$
a constant depending on $[\mA]_{0,\alpha}$ and $\mathfrak{m}_0$ as in the statement of Proposition \ref{p:monotonia freq}.

\begin{definition}\label{d:almost hom}
Let $\eta>0$ and let $u:B_1\to\R$ be a solution to thin obstacle problem \eqref{e:prob}. Assume that $x_0\in {\Gamma^{\textup{finite}}(u)\cap B'_{\sfrac12}}$ and $r\in(0,\sfrac12)$ is such that $J_u(x_0,r)$ is defined. Then, $u$ is called
\textit{$\eta$-almost homogeneous of \eqref{e:prob} in $B_r(x_0)$} if
\[
J_u(x_0,\sfrac{r}{2}) - J_u(x_0,\sfrac{r}{4})\leq \eta.
\]
\end{definition}

The following lemma justifies this terminology.

\begin{lemma}\label{l:almost hom}
For every $\eps>0$ and $\mathfrak m_0>0$, there exist $\eta, \varrho>0$ with the following property: 
if $u$ is a $\eta$-almost homogeneous solution in $ B_r(x_0) $ with $ r\leq \varrho$ and $x_0\in\Gamma^{\mathfrak{m}_0}(u)\cap B_{\sfrac12}'$, then
\begin{align}\label{e:almost hom}
\inf_{w\in \cH}\big\| (u_{\mA(x_0)})_{x_0,r} - w\big\|_{H^1(B_{1})} \leq \eps.
\end{align}
\end{lemma}

\begin{proof}
The proof follows by a contradiction argument similar to \cite[Lemma~5.5]{FS18-1}. 
Assume that for $ \eps>0 $ we could find sequences $ r_l $ of numbers and $u_l$ of $\frac1l$-almost homogeneous solutions in ${B_{r_l}(x_l)}$, such that
\begin{equation}\label{e.lontano}
 \inf_l\inf_{w\in\cH} 
 \big\| \big((u_l)_{{\mA(x_l)}}\big)_{x_l,r_l} - w\big\|_{H^1(B_{1})}\geq \eps\,,
 \end{equation}
with $ x_l \in \Gamma^{\textup{finite}}(u_l) \cap B_{\sfrac12}'$ and $\mathfrak m(x_l)\leq \mathfrak m_0$.
By Proposition \ref{c:compactness} there exists a subsequence (not relabeled) of $ v_l= \big((u_l)_{{\mA(x_l)}}\big)_{x_l,r_l}$ converging to a solution $ v_\infty $ of the thin obstacle problem in $ B_1$ for the standard Dirichlet energy. Moreover, we can assume that the points $ x_l $ converge to $ x_\infty \in \bar{B}_{\sfrac12}'$.
From Proposition \ref{p:doubling} we infer that
\[ 
H_{v_\infty} (\sfrac12) = \lim_l H_{v_l} (\sfrac12) \geq C \lim_l H_{v_l}(1) > 0\,,
\]
so that $ v_\infty $ is not zero.
On the other hand, we have that
\[ 
I_{v_\infty}(\sfrac12) - I_{v_\infty} (\sfrac 14)= \lim_l ( I_{v_l}(\sfrac12) - I_{v_l} (\sfrac 14))= 
\lim_l (J_u(x_l, \sfrac{r_l}{2})- J_u(x_l, \sfrac{r_l}{4})) =0. 
\]
This implies that $ v_\infty $ is a solution with constant frequency and thus is homogeneous (see for instance \cite[Proposition~2.7]{FS18-1}), contradicting \eqref{e.lontano}.
\end{proof}

A rigidity property of the type shown in \cite[Proposition~5.6]{FS18-1} holds in the case of non smooth coefficients as well. We call spine $S(w)$ of a function $w \in \cH$ the maximal subspace of invariance 
of $u$,
\[
S(w) := \Big\{ y\in\R^n\times \{0\}\;:\; w(x+y) = w(x) \quad \forall\; x\in \R^{n+1}\Big\}.
\]
We recall that the maximal dimension of the spine of a function $w$ in $\cH$ is at most $n-1$ {(cf. \cite[Section 5.2]{FS18-1})}, and we set $\cH^\top$ for the set of homogeneous solutions $w$
with $\dim S(w) = n-1$; whereas $\cH^\low:=\cH\setminus \cH^\top$.

\begin{proposition}\label{p:rigidity}
For every $\tau>0$ and $\mathfrak m_0>0$, there exists $\eta, \varrho>0$ with this property.
If $u$ is a $\eta$-almost homogeneous solution in $ B_r (x_0) $, $r\leq \varrho$ and $x_0\in\Gamma^{\mathfrak{m}_0}(u)\cap B_{\sfrac12}'$ with $\mathfrak m(x_0)\leq \mathfrak m_0$,
then the following dichotomy holds:
\begin{itemize}
\item[(i)] either for every point $x\in \Gamma^{\mathfrak{m}_0}(u)\cap B_{\sfrac r2}'(x_0)$ we have
\begin{align}\label{e:rigidity1}
\left\vert J_u(x,\sfrac r2)  - J_u(x_0,\sfrac r2)\right\vert\leq\tau,
\end{align}
\item[(ii)] or there exists a linear subspace $V\subset\R^{n}\times\{0\}$ of dimension $n-2$ such that
\begin{align}\label{e:rigidity2}
\begin{cases}
y\in\Gamma^{\mathfrak{m}_0}(u)\cap B_{\sfrac r2}'(x_0),\\
J_u(y,\sfrac r8) - J_u(y,\sfrac r{16})\leq \eta
\end{cases} 
\quad\Longrightarrow\quad \dist(y,x_0+V)\leq\tau r.
\end{align}
\end{itemize}
\end{proposition}
\begin{proof}
The proof proceeds by contradiction and follows the strategy in \cite[Proposition~5.6]{FS18-1}.
Let $\tau>0$ be a given constant and assume that there exist $r_l$ and a sequence $(u_l)_{l\in\N}$ of $\sfrac 1l$-almost homogeneous solutions in $B_{r_l}$ {(this clearly holds up to horizontal translations)} such that
 \begin{itemize}
 \item[(i)] there exists $x_l\in \Gamma^{\mathfrak{m}_0}(u_l)\cap B_{\sfrac{r_l}2}'$  for which
\begin{align}\label{e:rigidity1-contra}
\left\vert J_u(x_l , \sfrac{r_l}2) - J_u(\underline 0, \sfrac{r_l}2)\right\vert>\tau,
\end{align}
 \item[(ii)] for every linear subspace $V\in\R^n\times\{0\}$
 of dimension $n-2$ there exists  $y_l\in\Gamma^{\mathfrak{m}_0}(u_l)\cap  B_{\sfrac{r_l}2}'(x_0)$
(a priori depending on $V$) such that
 \begin{align}\label{e:rigidity2-contra}
 J_u(y_l, \sfrac{r_l}8) - J_u(y_l, \sfrac{r_l}{16})
 \leq\sfrac 1l
 \quad\text{and}\quad \dist\big(y_l, V\big) > \tau{r_l}.
 \end{align}
 \end{itemize}
We consider the rescaled functions $v_l:= (u_l)_{\underline 0,r_l}$. By the compactness result in Corollary~\ref{c:compactness} $ v_l $ converge, up to a subsequence, to a not zero solution to the thin obstacle problem with constant coefficients $ v_\infty $. In particular $v_\infty \in \mathcal{H}$ thanks to Lemma~\ref{l:almost hom}.  

If $v_\infty \in \cH^\top$, then \eqref{e:rigidity1-contra} is contradicted. Indeed, up to choosing a further subsequence, we can assume that $z_l:=r_l^{-1}x_l\to z_\infty \in \bar B_{\sfrac 12}$; moreover, $z_\infty$ is a critical point for $v_\infty$, because both $v_l(z_l) = |\nabla v_l(z_l)|=0$ and the convergence is $C^1$, and by a simple change of variables we have that 
\begin{align*}
\left\vert I_{v_\infty}(z_\infty,\sfrac12) - 
 I_{v_\infty}(0,\sfrac12)\right\vert
  &=\lim_{l\to\infty}\left\vert J_u(x_l, \sfrac {r_l}2) - J_u(0, \sfrac{r_l}2)\right\vert\geq \tau,
\end{align*}
 which is a contradiction to the constancy of the frequency at critical points of homogeneous solutions $v_\infty\in \cH^{\top}$ (see \cite[Lemma~5.3]{FS18-1}).

On the other hand, if $v_\infty\in\cH^\low$, we show a contradiction to {the second condition in }\eqref{e:rigidity2-contra} with $V$ any $(n-2)$-dimensional subspace containing $S(v_\infty)$.
Indeed, let $y_l$ be as in \eqref{e:rigidity2-contra} for such a choice of $V$. By compactness, up to passing to a subsequence (not relabeled), $z_l:=r_l^{-1}y_l\to z_\infty$ for some $z_\infty\in \bar B_{\sfrac12}$ with $\dist(z_\infty, V)\geq \tau>0$. 
Arguing as before, we obtain
 \begin{align*}
\left\vert I_{v_\infty}\big(z_\infty,\sfrac18\big) - 
 I_{v_\infty}\big(z_\infty,\sfrac1{16}\big)\right\vert 
 &=\lim_{l\to\infty}\left\vert J_u(y_l,\sfrac{r_l}8) - J_u(y_l,\sfrac{r_l}{16})\right\vert\leq \lim_{l\to\infty} \sfrac1l =0\,,
 \end{align*}
where the last inequality is given by the $\sfrac1l$-almost homogeneity of the functions $u_l$ {cf. the first condition in \eqref{e:rigidity2-contra})}.
Using \cite[Proposition~2.7, Lemma~5.2]{FS18-1} it follows that $z_\infty \in S(v_\infty)$, from which we infer a contradiction as  $z_\infty \in S(v_\infty)\subseteq V$ and $\dist(z_\infty,V)\geq \tau$.
\end{proof}

\subsection*{Proof of Theorem \ref{t:main}}
The proof is now a simple consequence of the results established in the previous sections. Indeed, we can follow verbatim \cite[Section 6]{FS18-1} (see also \cite[\S 5.3]{FS22}).
Recall that
\[
\Gamma^{\textup{finite}}(u):= \left\{x\in B_1' : \limsup_{r\to 0^+} N_u(x,r)<+\infty\right\}.
\]
By a simple rescaling argument, if $x_0\in B_1'$ 
and $r< \dist(x_0, \partial B_1)$, then the function $u_r(y):=u(x_0+ry)$ solves a thin obstacle problem \eqref{e:minimization} with $\mathbb{A}$ satisfying (H1) - (H3) and 
\[
\Gamma^{\textup{finite}}(u_r)\cap B'_{\frac12}= \bigcup_{\mathfrak{m}_0\geq \frac{3}{2}} \Gamma^{\mathfrak{m}_0}(u_r).
\]
Therefore, it is enough to show that $\Gamma^{\mathfrak{m}_0}(u) \cap B'_{\sfrac12}$ is rectifiable. To this aim we fix $\rho_0>0$ such that the conclusions of all propositions in the previous sections hold for points $x\in \Gamma^{\mathfrak{m}_0}(u) \cap B'_{\sfrac12}$ and radii $\rho\leq \rho_0$.
We can then follow the proof of \cite[Section 6]{FS18-1} applied to the intrinsic frequency $N_u$ starting at $\rho_0$: indeed, the proof uses only the lower bound of the frequency (cf. Corollary \ref{c:freq ben def}), the estimate of the spatial oscillation of the frequency in terms of the mean-flatness (cf. Proposition~\ref{p:mean-flatness vs freq bis}) and the rigidity of Proposition~\ref{p:rigidity}, together with the Reifenberg-type rectifiability criteria provided in the work by Naber and Valtorta \cite{NaVa1}.

Finally, we note that the proof of the rectifiability also gives the local finiteness of the measure of each $\Gamma^{\mathfrak{m}_0}(u)$, which we will use for the proof of
Theorem~\ref{t:lip} in the next section.

\section{Finiteness of the frequency for $ \mA \in W^{1,\infty} $}\label{s:A lipschitz}
In this section we prove the finiteness of the intrinsic frequency {$N_u$} at all free boundary points for a solution $u$ of \eqref{e:ob-pb local} assuming that the matrix field $\mA$ satisfies
(H1) with $p=\infty$, (H2), and (H3) (see also \cite{GS14}). Given this for granted, Theorem \ref{t:lip} is then an immediate consequence of Theorem \ref{t:main}.

We first establish several auxuliary results under the simplifying assumptions that the 
base point is the origin and that $\mA(\underline 0)=\Id$ in the spirit of \cite[Section 3.2]{FoGeSp15}.
Consider then the function $\mu:B_1\to[0,\infty)$ defined by
\[
\mu(x):=\langle \mA(x) \nu(x),\nu (x)\rangle \;\mbox{ if } x\neq 0 \; \mbox{ and }\;\mu(\underline{0})=1,
\]
where $ \nu(x)=\frac{x}{|x|} $. {Recalling that $\mA$ is Lipschitz continuous we infer that} 
$ \mu \in C^{0,1}(B_1)$, and 
\begin{equation}\label{e:bound mu}
{\lambda\leq \mu(x)\leq \Lambda, \quad \mbox{for every } x\in B_1,}
\end{equation}
{where $\lambda,\,\Lambda$ are the ellipticity constants in (H2)}
(for a proof see \cite[Lemma 3.10]{FoGeSp15}). 

Here, for the sake of simplicity, we follow the computations in \cite{FoGeSp15}
which use Almgren original frequency function (cf. \cite{Alm00}) tailored for
Lipschitz coefficients. Let us define the functions
\begin{equation}\label{def EH}
\Eco(r):=\int_{B_{r}} \langle \mA\nabla u, \nabla u \rangle \d x \quad \mbox{ and }  \quad  \Hco(r):=\int_{\de B_r} \mu u^2 \d \cH^n,
\end{equation} 
and the {energy driven} frequency function
\begin{equation}\label{def I}
\Ico(r):=\frac{r\Eco(r)}{\Hco(r)}.
\end{equation}
It is useful for the sequel to observe that
\begin{equation}\label{e:Eco equiv}
\Eco(r)=\int_{\de B_r}u\langle\mA\nabla u,\nu\rangle\d \cH^n\,.
\end{equation}
This equality follows by computing the divergence of the vector field $u\mA\nabla u$, by taking into account 
\eqref{e:ob-pb local}, and by exploiting the Signorini's ambiguous conditions together with (H3).

In order to establish the monotonicity of $\Ico$ we start with the following lemma.

\begin{lemma}\label{TrA-(n+1)mu}
{Let $\mA$ satisfies (H1) with $p=\infty$, (H2), and $\mA(\underline 0)=\Id$,
let $\mu $ be as above}. Then, there exists a constant $C>0$ {depending on $n$ and on $[\mA]_{0,1}$} such that for every  $r\in (0,1)$ and  $x\in B_r$, we have that 
\[ 
\left| \Tr \mA(x)-(n+1)\mu(x)\right| \leq Cr.  
\]
\begin{proof}
Fixed a point $x\in B_r$, let $\{\lambda_i\}_{i=1}^{n+1}$, be the eigenvalues of the matrix $\mA(x)$ 
and $\left\lbrace  e_i \right\rbrace_{i=1}^{n+1}$ be the corresponding orthonormal base of eigenvectors. 
Set $y_i := r\,e_i$,
then,
\begin{equation*}
\begin{split}
|\Tr \mA(x)-&(n+1)\mu(x)|=\left|\sum_{i=1}^{n+1} (\lambda_i-\mu(x))\right|=\left|\sum_{i=1}^{n+1}\Big(\langle \mA(x) e_i,e_i\rangle- \langle \mA(x) \nu(x), \nu(x)\rangle\Big)\right|\\
&\leq \sum_{i=1}^{n+1}\left(\Big| \langle \mA(x) e_i,e_i\rangle-\langle \mA(y_i) e_i, e_i\rangle\Big| + \Big| \langle \mA(y_i) e_i, e_i\rangle- \langle \mA(x) \nu(x), \nu(x)\rangle \Big|\right) \\
&= \sum_{i=1}^{n+1}\left(\big| \langle \left(\mA(x) -\mA(y_i)\right) e_i,e_i\rangle \big| + \big| \mu(y_i)-\mu(x) \big|\right) \\
&\leq C \sum_{i=1}^{n+1} (|e_i|^2+ 1)|x-y_i| \leq Cr,
\end{split}
\end{equation*}
where we used the Lipschitz continuity of $\mA$ and $\mu$, and that $x,\, y_i\in \overline B_r$.
\end{proof}
\end{lemma}

\begin{remark}
From Lemma \ref{TrA-(n+1)mu} we deduce that 
\[ 
\Tr \mA (x) -(n+1)\mu(x) \geq -Cr\,, 
\]
in turn implying for every $x\in B_1$
\begin{equation}\label{TrA/mu}
\mu^{-1}(x)\Tr \mA(x) \geq -Cr\mu^{-1}(x) + (n+1).
\end{equation}
\end{remark}

{First, we compute the derivative of $\Eco$.}
\begin{proposition}\label{E'}
{Let $u$ be a solution to \eqref{e:prob} under assumptions (H1)-(H3) with $p=\infty$, and 
$\mA(\underline 0)=\Id$, and let $\mu $ be as above. Then, there exists a constant $ C>0$ 
depending on $n$, $\lambda$, $\Lambda$, and $[\mA]_{0,1}$ such that
for $\mathcal L^1$-a.e. $r\in(0,1)$}
\begin{equation}
\Eco'(r)=2\int_{\de B_r} \mu^{-1} \langle  \mA\nu, \nabla u\rangle^2 \d \cH^n+ E_r,
\end{equation}
with
\[ 
E_r\geq -C\Eco(r) +\frac{n-1}{r}\Eco(r).
\]
\begin{proof}
By the coarea formula  and \cite[Lemma 3.4]{FoGeSp15} {applied to the Lipschitz vector field 
${\bf F}(x):=\frac{\mA(x)x}{r\mu(x)}$,} we have
\begin{equation}\label{derE}
\begin{split}
\Eco'(r) =&\int_{\de B_r} \langle \mA\nabla u, \nabla u \rangle \d \cH^n\\
 =& 2\int_{\de B_r} \mu^{-1} \langle  \mA\nu, \nabla u\rangle^2 \d \cH^n 
+\frac{1}{r}\int_{B_{r}}\mu^{-1} \nabla \mA \colon \mA x \otimes \nabla u \otimes \nabla u \d x\\ 
&+ \frac{1}{r}\int_{B_{r}} \langle \mA\nabla u, \nabla u\rangle \div \left(\mu^{-1}\mA x  \right) \d x
-\frac{2}{r}\int_{B_{r}} \langle  \mA\nabla u, \nabla^T \left( \mu^{-1}\mA x\right) \nabla u \rangle  \d x\\
 =:& 2\int_{\de B_r} \mu^{-1} \langle  \mA\nu, \nabla u\rangle^2 \d \cH^n +R_1 +R_2 +R_3.
\end{split}
\end{equation}
We now estimate the $R_i$'s. {We start with $R_1$. By using the Lipschitz continuity of $\mA$,
\eqref{e:bound mu} and (H2) we get}
\begin{align}\label{R1}
|R_1| 
&\leq \frac{1}{r}\int_{B_{r}} \sum_{i,j,k,l} \left| \mu^{-1}\partial_i a_{j,l} a_{i,k} x_k \; \partial_j u  \;\partial_l u \right| \d x\notag \\
&\leq \lambda^{-1} \int_{B_{r}}  \sum_{i,j,k,l} |\partial_i a_{j,l}|  | a_{i,k} \partial_j u \partial_l u| \d x 
\leq C \int_{B_{r}} \langle \mA \nabla u, \nabla u \rangle\d x=C\Eco(r).
\end{align}
By computing explicitly the divergence, $R_2$ rewrites as
\begin{align}\label{R2}
R_2&
=\frac{1}{r}\int_{B_{r}} \langle \mA\nabla u, \nabla u\rangle \sum_{i,j}^{n+1} \partial_i\left(\mu^{-1}a_{ij} x_j \right) \d x\notag\\
&=\frac{1}{r}\int_{B_{r}} \langle \mA\nabla u, \nabla u\rangle \sum_{i,j}^{n+1} \partial_i(
\mu^{-1}a_{ij}) x_j \d x + \frac{1}{r}\int_{B_{r}} \langle \mA\nabla u, \nabla u\rangle \mu^{-1}\Tr\mA \d x\notag\\
&\geq -C\int_{B_{r}} \langle \mA\nabla u, \nabla u\rangle \d x + 
\frac{1}{r}\int_{B_{r}}\langle \mA\nabla u, \nabla u\rangle \mu^{-1} \Tr\mA \d x\notag\\
&\geq -C\Eco(r)+\frac{n+1}{r}\Eco(r),
\end{align}
where we used the Lipschitz continuity of $\mu^{-1} \mA$, \MF{\eqref{e:bound mu}}, 
and \eqref{TrA/mu}. Analogously, for $ R_3 $ we have
\begin{align}\label{R3}
R_3 
& =-\frac{2}{r}\int_{B_{r}} \langle  \mA\nabla u, {(\nabla^T(\mu^{-1}\mA)x)\nabla u} \rangle  \d x-\frac{2}{r}\int_{B_{r}} \langle  \mA\nabla u, \mu^{-1}\mA \nabla u \rangle  \d x\notag\\
& \geq-C\int_{B_{r}} \langle  \mA\nabla u,\nabla u \rangle  \d x
-\frac{2}{r}\int_{B_{r}} \langle  \mA\nabla u,\left( \mu^{-1}\mA-\Id\right) \nabla u \rangle -\frac{2}{r}\int_{B_{r}}\langle  \mA\nabla u,\nabla u \rangle  \d x\notag \\
&\geq -C\Eco(r)-\frac{2}{r}\Eco(r),
\end{align}
where we used the Lipschitz continuity of $\mu^{-1} \mA$, (H2), \eqref{e:bound mu}, and $\mu^{-1}(\underline{0}) \mA(\underline{0})=\Id $.
Collecting \eqref{derE}-\eqref{R3} we conclude.
\end{proof}
\end{proposition}

We now focus on the derivative of $ \Hco $.

\begin{proposition}\label{H'}
{Let $u$ be a solution to \eqref{e:prob} under assumptions (H1)-(H3) with $p=\infty$, and 
$\mA(\underline 0)=\Id$, and let $\mu $ be as above. Then, there exists a constant $ C>0$ 
depending on $n$, $\lambda$, $\Lambda$, and $[\mA]_{0,1}$ such that
for $\mathcal L^1$-a.e. $r\in(0,1)$}
\begin{equation}
\Hco'(r)=\frac{n}{r}\Hco(r)+2\int_{\de B_r} u\langle \mA \nabla u, \nu\rangle \d \cH^n + H_r,
\end{equation}
with 
\[ 
|H_r|\leq C\Hco(r). 
\]

\begin{proof}
First note that by the definition of $\mu$, $\nu$, the divergence theorem implies
\[ 
\Hco(r)=\frac{1}{r}\int_{B_{r}} \div (u^2\mA x) \d x. 
\]
Thus, the coarea formula {and Lemma \ref{TrA-(n+1)mu} yield for $\cL^1$-a.e. $r\in(0,1)$}
\begin{equation}\label{derH}
\begin{split}
\Hco'(r)&= -\frac{1}{r}\Hco(r) + \frac{1}{r}\int_{\de B_r} \div \left( u^2 \mA x\right) \d \cH^n\\
&= -\frac{1}{r}\Hco(r) + \frac{2}{r}\int_{\de B_r} u \langle \mA x, \nabla u \rangle \d \cH^n 
+ \frac{1}{r}\int_{\de B_r} u^2 \Big( \sum_{i,j}^{n+1}\partial_i a_{i,j} x_j + \Tr\mA\Big) \d \cH^n\\
&=\frac{n}{r}\Hco+2\int_{\de B_r} u \langle \mA \nu , \nabla u \rangle  \d \cH^n + \frac{1}{r}\int_{\de B_r} u^2 \sum_{i,j}^{n+1}\partial_i a_{i,j} x_j \d \cH^n\\
&+\frac{1}{r}\int_{\de B_r} u^2\left( \Tr\mA-(n+1)\mu \right) \d \cH^n.
\end{split}
\end{equation}
We now estimate the last two summands. {Thanks to the Lipschitz continuity of $\mA$ and 
\eqref{e:bound mu} we have}
\begin{equation}
\left| \frac{1}{r}\int_{\de B_r} u^2 \sum_{i,j}^{n+1}\partial_i a_{i,j} x_j \d \cH^n\right| 
\leq C\int_{\de B_r} u^2 \d \cH^n \leq C\Hco(r)\,.
\end{equation}
Moreover, using Lemma \ref{TrA-(n+1)mu} we have that 
\begin{equation}
\left| \frac{1}{r}\int_{\de B_r} u^2\left( \Tr \mA-(n+1)\mu \right) \d \cH^n\right| 
\leq C\int_{\de B_r} u^2 \d \cH^n \leq C\Hco(r)\,.
\end{equation}
The conclusion then follows at once.
\end{proof}
\end{proposition}

We now prove the quasi monotonicity of $ \Ico $.

\begin{proposition}\label{I'}
{Let $\mA$ satisfies (H1) with $p=\infty$, (H2), (H3), and $\mA(\underline 0)=\Id$,
let $\mu $ be as above. Then, there exists a constant $ C>0$ 
depending on $n$, $\lambda$, $\Lambda$, and $[\mA]_{0,1}$}
such that the function
\[ 
(0,1]\ni r \mapsto e^{Cr} \Ico(r) 
\]
is non-decreasing, where we recall that $\Ico(r)=\frac{r\Eco(r)}{\Hco(r)}$.
\end{proposition}

\begin{proof}
Propositions \ref{E'} and \ref{H'}, formula \eqref{e:Eco equiv}, and the Cauchy-Schwarz inequality give
{for $\cL^1$-a.e. $r\in(0,1)$}
\begin{equation*}
\begin{split}
\Ico'(r)&=\frac{\d}{\d r}\left(\frac{r\Eco(r)}{\Hco(r)}\right)=\frac{\Ico(r)}{r}+r\frac{\Eco'(r)\Hco(r)-\Eco(r)\Hco'(r)}{\Hco^2(r)}\\
&=\frac{\Ico(r)}{r}+\frac{r}{\Hco^2(r)}\left( 2\Hco(r)\int_{\de B_r} \mu^{-1}\langle\mA\nu, \nabla u\rangle^2 \d \cH^n 
- 2\left(\int_{\de B_r} u\langle \mA \nabla u, \nu\rangle \d \cH^n  \right)^2 \right)\\
&\quad +\frac{r}{\Hco^2(r)} \left( E_r \Hco(r) -\frac{n}{r} \Hco(r)\Eco(r) -H_r\Eco(r) \right)\\
&\geq \frac{\Ico(r)}{r}+\frac{r}{\Hco^2(r)}\left ( E_r \Hco(r) -\frac{n}{r} \Hco(r)\Eco(r) -H_r\Eco(r) \right)\\
&\geq \frac{\Ico(r)}{r}+\frac{r}{\Hco^2(r)} ( -C\Eco(r)\Hco(r)-\frac{1}{r} \Hco(r)\Eco(r))=-C \Ico(r)\,.
\end{split}
\end{equation*}
The conclusion then follows at once.
\end{proof}

The quasi-monotonicity of $\Ico$ is exploited in what follows to show the finiteness of the intrinsic frequency $N_u$. To this aim we will also need the following auxiliary result.
	
\begin{lemma}\label{MonH}
{Let $u$ be a solution to \eqref{e:prob} under assumptions (H1)-(H3) with $p=\infty$, and 
$\mA(\underline 0)=\Id$.} Then there exists $\beta>0$ such that
\[
\frac{\Hco(t)}{t^\beta}e^{Ct}\leq \frac{\Hco(r)}{r^\beta}e^{Cr}\qquad \forall\;0<r<t<1.
\]
\end{lemma}

\begin{proof}
From Proposition \ref{H'} and \eqref{e:Eco equiv} we have for $\cL^1$-a.e. $r\in(0,1)$
\begin{equation*}
\Hco'(r)\leq \frac{n}{r}\Hco(r)+2\Eco(r) +C\Hco(r)\,,
\end{equation*}
so that
\begin{equation*}
\frac{\Hco'(r)}{\Hco(r)}\leq \frac{n}{r} +2\frac{\Ico(r)}{r} +C\leq \frac{n+C\Ico(1)}{r}+C,
\end{equation*}
from which we obtain
\begin{equation}
\frac{\d}{\d r}\left(\ln \left(\frac{\Hco(r)}{r^\beta}\right)\right)\leq C
\end{equation}
where $ \beta= n+C\Ico(1)$.
The conclusion then follows by a simple integration.
\end{proof}

We can finally prove the finiteness of $N_u(x_0,0^+)$ for every point $x_0$ in $\Gamma(u)$.

\begin{proposition}\label{I limitata}
{Let $u$ be a solution to \eqref{e:prob} under assumptions (H1)-(H3) with $p=\infty$.} 
Then, $\Gamma^{\textup{finite}}(u) = \Gamma(u)$, i.e. ${N_u}(x_0,0^+)<\infty$ for every 
$x_0\in \Gamma(u)$.
\end{proposition}

\begin{proof}
Without loss of generality we verify the finiteness of the frequency in $x_0=\underline 0\in \Gamma(u)$. Moreover, by the arguments in \S\ref{s.intrinsic frequency} it is enough to consider the case $\mA(\underline 0)= \Id$: indeed, the intrinsic frequency function is defined after change of coordinates $\Phi_{x_0}$ which sets the matrix $\mathbb{A}(0)$ to be the identity.
In the sequel we use the convention adopted throughout the paper to drop the base point being equal to the origin.

We begin by estimating $ H_u(r) $ and $ D_u(r) $ in terms of $ \Hco(r) $ and $ \Eco(r) $, respectively.
Let us begin with $ H_u(r) $: from Lemma \ref{MonH} we get
\begin{align}\label{H>Hco}
H_u(r)&=\int -\phi'\Big({\textstyle{\frac{|x|}{r}}}\Big)\frac{u^2}{|x|}\d x= \int_{\sfrac r2}^{r} \int_{\de B_s}-\phi'\Big({\textstyle{\frac{|x|}{r}}}\Big)\frac{u^2}{s}\;  \d \cH^n\d s \notag\\
& \geq \int_{\sfrac r2}^{r} -\phi'\Big({\textstyle{\frac{s}{r}}}\Big) \frac{1}{\Lambda s} \int_{\de B_s} u^2 \mu \;\d \cH^n \d s = \int_{\sfrac r2}^{r} -\phi'\Big({\textstyle{\frac{s}{r}}}\Big) \frac{1}{\Lambda s}\Hco(s) \;\d s \notag\\
&\geq \int_{\sfrac r2}^{r} -\phi'\Big({\textstyle{\frac{|x|}{r}}}\Big) \frac{1}{\Lambda s}\Hco(r) e^{C(r-s)} \frac{s^\beta}{r^\beta} \;\d s\notag\\
&\geq C \Hco(r) \int_{\sfrac r2}^{r} -\phi'\Big({\textstyle{\frac{|x|}{r}}}\Big) \frac{1}{s} \d s\geq  C \Hco(r).
\end{align}
Instead, for $ \Eco $ and $ D_u $ we have
\begin{equation}\label{D<Eco}
D_u(r)=\int_{B_{r}} \phi\Big({\textstyle{\frac{|x|}{r}}}\Big) |\nabla u|^2 \d x
\leq \int_{B_r}|\nabla u|^2 \d x \leq 
\lambda^{-1}\Eco(r).
\end{equation}
Thus, from \eqref{H>Hco}, \eqref{D<Eco} and the definition of the frequency $ I_u(r) $ we conclude 
by taking into account Proposition \ref{I'}
\[ 
{N_u(r)=}I_u(r)=\frac{rD_u(r)}{H_u(r)}\leq C\frac{r\Eco(r)}{\Hco(r)}=C\Ico(r)\leq C\Ico(1)<\infty. \qedhere
\]
We then conclude that all points of the free boundary have finite frequency, i.e.
$\Gamma(u) = \Gamma^{\textup{finite}}(u)$.
\end{proof}

We are then in the position to prove Theorem \ref{t:lip}.

\begin{proof}[Proof of Theorem \ref{t:lip}]
The rectifiability of the free boundary is a consequence of Theorem \ref{t:main} and the previous Proposition \ref{I limitata} which establishes that all free boundary points belongs to $\Gamma^{\textup{finite}}(u)$.

In order to deduce the local finiteness of the Minkowski content of $\Gamma(u)$, we observe that the intrinsic frequency is locally bounded in $B_{\frac12}'\cap \Gamma(u)$:
i.e., there exists $\mathfrak{m}_0>0$ such that
\[
B_{\frac12}'\cap \Gamma(u) \subset \Gamma^{\mathfrak{m}_0}(u).
\]
Indeed, we have that $N_u(x, r)\leq C\,\Ico(x, r)\leq C\Ico(x,\sfrac12)$ for every $r\in (0,\frac12]$; tanking into account the continuity of $\overline B_{\sfrac12}'\cap\Gamma(u)
\ni x\to \Ico(x,\sfrac12)$, we infer that $N_u(x, r)$ is bounded in $\overline B_{\sfrac12}'\cap \Gamma(u)$ for every $r\in (0,\sfrac12]$.

By simple covering ad scaling arguments, the conclusion of Theorem \ref{t:lip} is shown for every compact $K\subset\subset B_1'$.
\end{proof}

\subsection{Free boundary of nonlinear thin obstacle problems}\label{ss:nonlinear}
The results proven above can be applied to the case of nonlinear thin obstacle problems studied in \cite{AAS24}, i.e. to the class of problems
\begin{equation}\label{min}
 \min_{u\in \mathcal{A}} \; \int_{B_{1}^+} f\left(\nabla u\right) \d x ,
\end{equation}
where the energy density $f:\R^{n+1}\to \R$ is convex and is of the form
\begin{equation}\label{formula-f}
f(p)=h(|p|)
\end{equation}
for every $p \in \mathbb{R}^{n+1}$, and the matrix $D_p^2 f(p)$ is uniformly coercive
on compact subsets, i.e. fulfills the following local ellipticity condition: for every
$M>0$ there exists $\lambda=\lambda(M)>0$ such that
\begin{equation}\label{ellipticity}
\langle D_p^2 f(p)\xi, \xi\rangle \geq \lambda |\xi|^2
\end{equation}
for every $|p|\leq M$ and  $\xi\in \mathbb{R}^{n+1}$.

As shown in \cite{AAS24}, if the function {$h\in C^2(\R)$} satisfies
\begin{equation*}
h(0)=h'(0)=0, \quad  h''(t) = 1 + O(t) \qquad\textup{for } t\to 0^+,
\end{equation*}
then the solutions to the variational problem in \eqref{min} are
$C^{1, \sfrac12}_{\textup{loc}}(B_1^+\cup B'_1)$.
Here we show that, if in addition
\begin{equation}\label{e.h''}
h''(t) = 1 + O(t^2) \qquad\textup{for } t\to 0^+,
\end{equation}
and, for the sake of simplicity, $h\in C^\infty$,
then we may apply Theorem~\ref{t:lip} to infer all the results on
the free boundary regularity in that statement.

\begin{proposition}
Let {$u\in W^{1,\infty}_{\loc}(B_1^+)$} be a solution to \eqref{min} in $\mathcal{A}$ with $h\in C^\infty$ under the assumptions \eqref{formula-f}, \eqref{ellipticity} and \eqref{e.h''},
then $u\in C^{1, \sfrac12}_{\textup{loc}}(B_1^+\cup B'_1)$, and the free boundary $\Gamma(u)$
is $(n-1)$-rectifiable and its Minkowski content is locally finite, i.e. for every $K\subset\subset B_1'$ there exists a constant $C(K)>0$ such that
\[
\mathcal{L}^{n+1}\big(\mathcal{T}_r(\Gamma(u)\cap K)\big) \leq C(K) r^2,
\]
for every $r\in (0,1)$.

Moreover, there exists a set $\Sigma(u)\subset \Gamma(u)$
with Hausdorff dimension at most $n-2$ such that for every
$x\in\Gamma(u)\setminus\Sigma(u)$
\[
N_u(x,0^+)\in\{2m,2m-\sfrac12,2m+1\}_{m\in\mathbb N\setminus\{0\}}\,.
\]
\end{proposition}

\begin{proof}
The solution to \eqref{min} is $C^{1,\sfrac12}_{\textup{loc}}(B_1^+\cup B'_1)$
by \cite{AAS24} and by standard elliptic regularity $u\in C^\infty(B_1^+)$ (thanks to the simplifying assumption $h\in C^\infty$).
Moreover, $u$ can be characterized as the weak solution to the system
\begin{align}\label{Prob}
\begin{cases}
	&\textup{div} (\nabla_p f(\nabla u))=0  \mbox{ in } B^{+}_1 \\ 
	&u\, \partial_{n+1} f (\nabla u) =0  \mbox{ on } B'_1 \\
	&-\partial_{n+1} f (\nabla u) \geq 0 \mbox{ on } B'_1\\
	&u\geq 0 \mbox{ on } B'_1\\
	&u= g \mbox{ on } (\partial B_1)^+
	\end{cases}.
 \end{align}
In particular, we deduce from the first equation in \eqref{Prob} 
that for every {$\phi \in C^1(B_1^+)$ with support non-intersecting $(\partial B_1)^+$}
\begin{equation}\label{e.eulero}
\int_{B_1^+} \langle\nabla_p f(\nabla u), \nabla \phi(x)\rangle \, dx =0 \,.
\end{equation}
We can write assumption \eqref{e.h''} in the form $h''(t) = 1 +\omega(t)$ with
$|\omega(t)|\leq C\, t^2$ for $t$ sufficiently small. Integrating, we infer that
\[
h'(t) = t \big(1 + \tilde \omega(t)\big)\,,\qquad \tilde \omega(t) = \frac{1}{t}\int_{0}^t \omega(s)\, ds\,,
\]
and {$\tilde \omega \in C^1$}
with $\tilde \omega(0) = \tilde \omega'(0) = 0$ and
\[
\tilde \omega'(t) = -\frac1{t^2}\int_0^t\omega(s)\, ds + \frac1t\, \omega(t)\quad \Longrightarrow \quad |\tilde \omega'(t)|\leq C\, t,
\]
for $t$ sufficiently small. The first variations \eqref{e.eulero} reads then for every
{$\phi \in C^1(B_1^+)$ with support non-intersecting $(\partial B_1)^+$}
\begin{equation*}
\int_{B_1^+} \langle\big(1+\tilde \omega(|\nabla u(x)|)\big)
\nabla u(x),\nabla \phi(x)\rangle \, dx =0 \,,
\end{equation*}
which is the Euler-Lagrange equation of the linear thin obstacle problem driven by the quadratic energy
\[
\int_{B_1^+} \theta(x)\, |\nabla u(x)|^2\, dx, \qquad
\theta(x) := 1+\tilde \omega(|\nabla u(x)|).
\]
{Note that $1\leq \theta(x)\leq 1+C|\nabla u(x)|$, therefore $\theta$ is
locally bounded on $B_1^+\cup B'_1$. Thus,} if we prove that the function
$\theta$ is {locally} Lipschitz continuous {on $B_1^+\cup B'_1$},
we can apply Theorem~\ref{t:lip} and conclude all the results about the
structure of $\Gamma(u)$.

To this aim, we notice that for nontrivial solutions $u$ in $B_1^+$ we have 
\[
\nabla \theta(x)  = \tilde \omega'(|\nabla u(x)|)\, D^2u(x) \frac{\nabla u(x)}{|\nabla u(x)|}\quad\textup{if $|\nabla u(x)|\neq 0$.}
\]
Extend for simplicity $u$ by even reflection to the whole $B_1$ (without renaming the function $u$) and let $d:B_1\to [0,\infty)$ be the distance from the free boundary $\Gamma(u)$.
By \eqref{Prob} the function $u$ satisfies the nonlinear elliptic equation
\[
\textup{div} \left(\nabla_p f(\nabla u)\right)=0\qquad \mbox{ on } B_1\setminus \{(x',0): u(x',0)=0\}\,,
\]
and therefore the following classical elliptic estimates hold {locally in $B_1\setminus \{(x',0): u(x',0)=0\}$}:
\begin{gather*}
|\nabla u(x)|\leq C\, d(x)^{-1} \|u\|_{L^{\infty}(B_{d(x)}(x))}\leq C\, d(x)^{\sfrac12},\\
|D^2 u(x)|\leq C\, d(x)^{-2} \|u\|_{L^{\infty}(B_{d(x)}(x))}\leq C\,d(x)^{-\sfrac12}.
\end{gather*}

Recalling that by assumption
\[
\tilde \omega'(|\nabla u(x)|)\leq C|\nabla u(x)|\,,
\]
we then conclude {for points outside the contact set, i.e. $B_1\setminus \{(x',0): u(x',0)=0\}$, that }
\begin{equation}\label{e.derivata theta}
|\nabla \theta(x)|  \leq \frac{|\tilde \omega'(|\nabla u(x)|)|}{|\nabla u(x)|} \, {|D^2u(x)|\, |\nabla u(x)|}\leq Cd(x)^{-\sfrac12}\,d(x)^{\sfrac12} \leq C\,.
\end{equation}
{Moreover, if $x\in\Gamma(u)$ then $|\nabla u(x)|=0$, so that
for every $y\in B_1^+\cup B_1'$
\begin{align*}
 |\theta(x)-\theta(y)|&=|\tilde\omega(|\nabla u(y)|)|\leq
 \int_0^1|\tilde\omega'(t|\nabla u(y)|)||\nabla u(y)|\,dt\\
 &\leq C|\nabla u(y)|^2  \leq C|x-y|,
\end{align*}
using the optimal regularity of $u$.
Finally, if $x$ belongs to the relative interior of $\{(x',0): u(x',0)=0\}$ in $B'_1$, we use the odd riflection across the hyperplane $\{x_{n+1}=0\}$ as in \cite[Theorem 4.1]{AAS24} to infer that \eqref{e.derivata theta} holds as well.

In conclusion, $\theta$ is locally Lipschitz continuous on $B_1^+\cup B'_1$.
}
\end{proof}

%


\appendix

\section{Order of contact}\label{a:order of contact}

We introduce the definition of lower and upper order of contact at zero in a point.
\begin{definition}
Let $v\in H^1(\Omega)$, $x_0\in\Omega\subset\R^{n+1}$, the lower and upper orders of contact with $0$ of $v$ at $x_0$ are 
defined respectively as 
\begin{align}
&\underline{\vartheta}(x_0):=\sup\Big\{\vartheta\in\R:\,\limsup_{\rho\to0^+}
\frac{H_v(x_0,\rho)}{\rho^{n+2\vartheta}}<\infty\Big\}\label{e:thetasotto},\\
&\overline{\vartheta}(x_0):=\inf\Big\{\vartheta\in\R:\,\liminf_{\rho\to0^+}
\frac{H_v(x_0,\rho)}{\rho^{n+2\vartheta}}>0\Big\}\,.\label{e:thetasopra}
\end{align}
\end{definition}
Few elementary properties of $\underline{\vartheta}(x_0)$ and $\overline{\vartheta}(x_0)$ 
are resumed in the ensuing list: for all $x_0\in\Omega$ we have
\begin{enumerate}
 \item  $-\infty\leq\underline{\vartheta}(x_0)\leq\overline{\vartheta}(x_0)\leq\infty$,
 \item 
 \[
\underline{\vartheta}(x_0)=\sup\Big\{\vartheta\in\R:\,\lim_{\rho\to0^+}
\frac{H_u(x_0,\rho)}{\rho^{n+2\vartheta}}=0\Big\}\,,
 \]
 \item 
 \[
\overline{\vartheta}(x_0)=\inf\Big\{\vartheta\in\R:\,\lim_{\rho\to0^+}
\frac{H_u(x_0,\rho)}{\rho^{n+2\vartheta}}=\infty\Big\}\,.
 \]
\end{enumerate}
Additionally, we compare the latter notions with those used by Koch, R\"uland and Shi \cite{KRS17}.
\begin{proposition}
Let $u\in \mathscr{A}$ be a solution to \eqref{e:minimization} under the hypotheses (H1) and (H2), and  $x_0\in \Gamma(u)$ with $\mA(x_0)=\Id$. Then, on setting $A_\rho(x_0):=B_\rho(x_0)\setminus B_{\sfrac\rho2}(x_0)$, we have that
\begin{align}
&\underline{\vartheta}(x_0)=
\liminf_{\rho\to0^+}\frac{\ln\Big(\fint_{A_\rho(x_0)}u^2\d x\Big)^{\sfrac12}}{\ln\rho}=:\underline{\kappa}(x_0)\,,
\label{e:underline kappa}\\
&\overline{\vartheta}(x_0)=\limsup_{\rho\to0^+}\frac{\ln\Big(\fint_{A_\rho(x_0)}u^2\d x\Big)^{\sfrac12}}{\ln\rho}
=:\overline{\kappa}(x_0)\,.\label{e:overline kappa}
\end{align}
\end{proposition}
\begin{proof}
 We shall only prove the equality in \eqref{e:underline kappa}, the other in \eqref{e:overline kappa} being completely analogous.
We first note that by Lemma \ref{l:monotonia H}
\begin{equation}\label{e:H vs energia anello}
\frac12\,\|\phi'\|_\infty^{-1}\, \rho\, H_u(x_0,\rho)\leq \int_{A_{\rho}(x_0)} |u(x)|^2\, \d x \stackrel{\eqref{e:L2 vs H}}{\leq} C\, \rho\, H_u(x_0,\rho),
\end{equation}
for points on the free boundary with $\mA(x_0)=\Id$.
Assume $\underline{\kappa}(x_0)\in\R$, then for every $\eps>0$ there are $\rho_\eps\in(0,1)$ 
 and $\rho_j\downarrow 0$ such that for all $\rho\in(0,\rho_\eps)$
 \[
  \fint_{A_\rho(x_0)}u^2\d x\leq\rho^{2(\underline{\kappa}(x_0)-\eps)}\,,
 \]
 and for all $j\in\N$
 \[
  \fint_{A_{\rho_j}(x_0)}u^2\d x\geq\rho_j^{2(\underline{\kappa}(x_0)+\eps)}\,.
 \]
From the former inequality and \eqref{e:H vs energia anello} 
we infer that $\underline{\kappa}(x_0)-\eps\leq\underline{\vartheta}(x_0)$, and thus 
$\underline{\kappa}(x_0)\leq\underline{\vartheta}(x_0)$. 
Instead, from the latter inequality and \eqref{e:H vs energia anello} we deduce that 
$\underline{\kappa}(x_0)+2\eps>\underline{\vartheta}(x_0)$, thus $\underline{\kappa}(x_0)\geq\underline{\vartheta}(x_0)$.
Therefore, $\underline{\kappa}(x_0)=\underline{\vartheta}(x_0)$.

If $\underline{\kappa}(x_0)=-\infty$ then there is $\rho_j\downarrow 0$ such that for all $i\in\N$ there is
$j_i\in\N$ such that for all $j\geq j_i$
\[
  \fint_{A_{\rho_j}(x_0)}u^2\d x\geq\rho_j^{-2i}\,,
\]
and thus $-i+\frac12>\underline{\vartheta}(x_0)$, in turn implying $\underline{\vartheta}(x_0)=-\infty$.

If $\underline{\kappa}(x_0)=\infty$ then for every $i\in\N$ there is $\rho_i\in(0,1)$ 
such that for all $\rho\in(0,\rho_i)$
 \[
  \fint_{A_\rho(x_0)}u^2\d x\leq\rho^{2i}\,,
 \]
from which we conclude that $\underline{\vartheta}(x_0)\geq i$, and thus $\underline{\vartheta}(x_0)=\infty$. 

In conclusion, $\underline{\vartheta}(x_0)=\underline{\kappa}(x_0)$ in all possible instances.
\end{proof}

For solutions to the thin obstacle problem the points with finite frequency are points with finite order of contact.

\begin{lemma}\label{l:order contact vs finite frequency}
Let $u$ be a solution to the thin obstacle problem \eqref{e:ob-pb local} in $B_1$. Then, for every $x_0\in \Gamma(u)$
\begin{equation}\label{e:I_u leq overline theta}
\limsup_{r\to 0^+}I_u(x_0,r)\geq\overline{\vartheta}(x_0)\,.  
 \end{equation}
Moreover,  if $\displaystyle{\limsup_{r\to 0^+}I_u(x_0,r)}<\infty$ then the limsup is actually a limit and
 \begin{equation}\label{e:order contact vs finite frequency}
\lim_{r\to 0^+}I_u(x_0,r)=\underline{\vartheta}(x_0)=\overline{\vartheta}(x_0)\in[\sfrac32,\infty)\,.
 \end{equation}
\end{lemma}

\begin{proof}
Without loss of generality we take $x_0=\underline{0}$, and set $\overline{I_u}(0^+):=\displaystyle{
\limsup_{r\to 0^+}I_u(\underline{0},r)}$.

We start off proving \eqref{e:I_u leq overline theta}. Without loss of generality we assume $\overline{I_u}(0^+)<\infty$, the inequality being trivial otherwise. Hence, the doubling of both $H_u(\underline{0},\cdot)$ and $D_u(\underline{0},\cdot)$ hold thanks to Proposition \ref{p:doubling}. 
Then, we use the equality in \eqref{e:variationHprime}, namely
\begin{align*}
H_u'(r)=\frac{n}{r}\,H_u(r)+2G_u(r)\,,
\end{align*}
and \eqref{e:variationD errore} with $x=x_0=\underline{0}$ and $\kappa=0$, to infer that 
\begin{align*}
\Big|H_u'(r)-\frac{n}{r}\,H_u(r)-2D_u(r)\Big|&\leq Cr^\alpha
\Big(D_u(r)+\frac1{r^{\sfrac12}}H_u^{\sfrac12}(r)D_u^{\sfrac12}(r)\Big)\\
&\leq Cr^\alpha D_u(r)\Big(1+I_u^{-\sfrac12}(r)\Big)\,,
\end{align*}
in turn implying 
\begin{align}\label{e:stima fondamentale lnH 0}
\Big|\frac{\d}{\d r}\ln\Big(\frac{H_u(r)}{r^n}\Big)-\frac2rI_u(r)\Big|\leq  Cr^{\alpha-1}I_u(r)(1+I_u^{-\sfrac12}(r))
=Cr^{\alpha-1}(I_u(r)+I_u^{\sfrac12}(r))\,.
\end{align}
Then, for every $\eps>0$ there is $r_\eps>0$ such that $I_u(r)\leq\overline{I_u}(0^+)+\eps$ for every $r\in(0,r_\eps)$.
We use \eqref{e:stima fondamentale lnH 0} to deduce for such radii that 
\begin{align*}
\frac{\d}{\d r}\ln\Big(\frac{H_u(r)}{r^n}\Big)
\leq \frac{2}r(\overline{I_u}(0^+)+\eps)+Cr^{\alpha-1}\,.
\end{align*}
Hence, by direct integration we get that for all $0<r<s<r_\eps$
\[
0<\frac{H_u(s)}{s^{n+2(\overline{I_u}(0^+)+\eps)}}e^{-\frac C\alpha s^\alpha}
\leq\frac{H_u(r)}{r^{n+2(\overline{I_u}(0^+)+\eps)}}
e^{-\frac C\alpha r^\alpha}\,.
\]
From this and the very definition of $\overline{\vartheta}(\underline{0})$ in \eqref{e:thetasopra} we have 
$\overline{\vartheta}(\underline{0})\leq\overline{I_u}(0^+)+\eps$ for every $\eps>0$, which implies \eqref{e:I_u leq overline theta}.

In oder to prove \eqref{e:order contact vs finite frequency} we combine the results in \eqref{e:I_u leq overline theta} 
with those in Proposition~\ref{p:monotonia freq} (cf. \eqref{e:almost monotonicity}) to infer that the $\limsup$ of the frequency is actually a limit, so that the latter rewrites as
\begin{equation}\label{e:order contact vs finite frequency ter}
\lim_{r\to 0^+}I_u(r)\geq\overline{\vartheta}(\underline{0})\,.
\end{equation} 
Therefore, arguing as above, by the inequality in \eqref{e:freq lb bis} of Corollary~\ref{c:quasi additive monotonicity} and 
\eqref{e:stima fondamentale lnH 0} we get that for every 
$r\in(0,r_\eps)$
\[
\frac{\d}{\d r}\ln\Big(\frac{H_u(r)}{r^n}\Big)\geq  
\frac{2}r(\overline{I_u}(0^+)-\eps)-Cr^{\alpha-1}\,,
\]
from which we conclude by integration that for all $0<r<s<r_\eps$
\[
0<\frac{H_u(r)}{r^{n+2(\overline{I_u}(0^+)-\eps)}}
e^{\frac C\alpha r^\alpha}\leq 
\frac{H_u(s)}{s^{n+2(\overline{I_u}(0^+)-\eps)}}e^{\frac C\alpha s^\alpha}\,.
\]
Hence, we deduce that $\overline{I_u}(0^+)-\eps\leq\underline{\vartheta}(\underline{0})$ for every $\eps>0$,
\eqref{e:order contact vs finite frequency} then follows at once
from the last inequality, \eqref{e:order contact vs finite frequency ter},  and Corollary \ref{c:quasi additive monotonicity}.
\end{proof}

As a consequence of a Carleman type estimate in \cite{KRS17} 
it is established there that for the solutions to the variable coefficients thin obstacle problem:
\begin{enumerate}
\item[(a)] $\underline{\vartheta}(x_0)=\overline{\vartheta}(x_0)$ for every $x_0\in B_1'$;
\item[(b)] if $\underline{\vartheta}(x_0)<\infty$, then doubling for $H_{u}(x_0,\cdot)$ holds provided $\mA(x_0)=\Id$.
\end{enumerate}
Items (a) and (b) right above yield the doubling of $H_{u_{\mA(x_0)}}(x_0,\cdot)$, in turn implying that for $D_{u_{\mA(x_0)}}(x_0,\cdot)$ thanks to an elementary Cacciopoli's inequality.
The latter and the proof of Proposition \ref{p:monotonia freq}
imply the quasi-monotonicity of $N_u(x_0, \cdot)=I_{u_{\mA(x_0)}}(x_0,\cdot)$ and thus the finiteness $N_u$.
On the other hand, item (a) of Lemma \ref{l:order contact vs finite frequency} shows that points with finite frequency have finite order of contact. Therefore we infer the following corollary.

\begin{corollary}
Let $u\in \mathscr{A}$ be a solution to \eqref{e:minimization} under the hypotheses (H1) and (H2).
Then, the subset of points of the free boundary with finite order of contact is well-defined
\begin{align*}
\Gamma^{\textup{finite}}(u) &= \Big\{x_0\in \Gamma(u) \,:\, \limsup_{r\to 0^+}N_u(x_0,r)<\infty\Big\}\\
&= \Big\{x_0\in \Gamma(u) \,:\, \underline{\kappa}(x_0)=\overline{\kappa}(x_0)<\infty\Big\}.
\end{align*}
In particular, the points with finite order of contact do not depend on the choice of the cut-off function $\phi$ in the definition of the frequency function.
\end{corollary}

%
%

\bibliographystyle{plain}

\end{document}